\documentclass[12pt]{amsart}%
\usepackage{amsfonts}
\usepackage{amsmath}
\usepackage{amssymb}
\usepackage{amscd}
\usepackage{graphicx}%
\providecommand{\U}[1]{\protect \rule{.1in}{.1in}}
\newtheorem{theorem}{Theorem}
\newtheorem{lemma}{Lemma}
\newtheorem{proposition}{Proposition}
\newtheorem{remark}{Remark}

\newtheorem{definition}{Definition}
\newtheorem{corollary}{Corollary}
\newtheorem{example}{Example}
\numberwithin{equation}{section}
\begin{document}
\title[Gradient Shrinking Sasaki-Ricci Solitons on Sasakian Manifolds]{Gradient Shrinking Sasaki-Ricci Solitons on Sasakian Manifolds of Dimension Up
to Seven}
\author{$^{^{\dag}}$Der-Chen Chang}
\address{$^{^{\dag}}$Department of Mathematics and Statistics, Georgetown University,
Washington D. C. 20057, USA\\
Graduate Institute of Business Administration, College of Management, Fu Jen
Catholic University, Taipei 242, Taiwan, R.O.C.}
\email{chang@georgetown.edu}
\author{$^{\ast}$Shu-Cheng Chang}
\address{$^{\ast}$Department of Mathematics, National Taiwan University, Taipei, Taiwan}
\email{scchang@math.ntu.edu.tw}
\author{$^{\dag}$Yingbo Han}
\address{$^{\dag}${School of Mathematics and Statistics, Xinyang Normal University}\\
Xinyang,464000, Henan, P.R. China}
\email{{yingbohan@163.com}}
\author{$^{\dag \dag}$Chien Lin}
\address{$^{\dag \dag}$Mathematical Science Research Center, Chongqing University of
Technology, 400054, Chongqing, P.R. China}
\email{chienlin@cqut.edu.cn}
\author{$^{\ast \ast}$Chin-Tung Wu}
\address{$^{\ast \ast}$Department of Applied Mathematics, National Pingtung University,
Pingtung 90003, Taiwan}
\email{ctwu@mail.nptu.edu.tw }
\thanks{$^{\dag}$Der-Chen Chang is partially supported by an NSF grant DMS-1408839 and
a McDevitt Endowment Fund at Georgetown University. $^{\ast} $Shu-Cheng Chang
and $^{\ast \ast}$Chin-Tung Wu are partially supported in part by the MOST of
Taiwan. $^{\dag}$Yingbo Han is partially supported by an NSFC 11971415 and
Nanhu Scholars Program for Young Scholars of {Xinyang Normal University}.
$^{\dag \dag}$Chien Lin is partially supported by the Project of the Ministry
of Science and Technology of China (Grant QN2022035003L)}

\begin{abstract}
In this paper, we show that the uniform $L^{4}$-bound of the transverse Ricci
curvature along the Sasaki-Ricci flow on a compact quasi-regular transverse
Fano Sasakian $(2n+1)$-manifold $M.$ When $M$ is dimension up to seven and the
space of leaves of the characteristic foliation is well-formed, we first show
that any solution of the Sasaki-Ricci flow converges in the Cheeger-Gromov
sense to the unique singular orbifold Sasaki-Ricci soliton on
$M_{\mathrm{\infty}\text{ }}$which is a $S^{1}$ -orbibundle over the unique
singular K\"{a}hler-Ricci soliton on a normal projective variety with
codimension two orbifold singularities. Secondly, for $n=1,$ we show that
there are only two nontrivial Sasaki-Ricci solitons on a compact quasi-regular
Fano Sasakian three-sphere with its leave space $Z_{m_{1}}$-teardrop and
$Z_{m_{1},m_{2}}$-football, respectively. For $n=2,3,$ we show that the
Sasaki-Ricci soliton is trivial one if $M$ is transverse $K$-stable.

\end{abstract}
\maketitle
\tableofcontents

\section{Introduction}

Sasakian geometry is very rich as the odd-dimensional analogous of K\"{a}hler
geometry. A Sasaki-Einstein $5$-manifold is to say that its K\"{a}hler cone is
a Calabi-Yau threefold. Such manifolds provide interesting examples of the
AdS/CFT correspondence. By the second structure theorem for Sasakian
manifolds, there is a $\mathbf{S}^{1}$-action with the Reeb vector field which
generates the finite isotropy groups. Only if the isotropy subgroup of every
point is trivial is the regular free action. They are regular Sasakian
manifolds and its leave spaces are smooth K\"{a}hler surfaces. In general, the
space of leaves has at least the codimension two fixed point set of every
non-trivial isotropy subgroup or the codimension one fixed point set of some
non-trivial isotropy subgroup. They are quasi-regular Sasakian manifolds and
its $\mathbf{S}^{1}$-fibrations will be orbifold surfaces with orbifold
singularities (cf. section $2$). Furthermore, there is a well-known
classification of \ compact Fano K\"{a}hler-Einstein smooth surfaces due to
Tian-Yau and then leads to a first classification of all compact regular
Sasaki-Einstein $5$-manifolds. There are quasi-regular Sasaki-Einstein metrics
on connected sums of $\mathbf{S}^{2}\times \mathbf{S}^{3},$ rational homology
$5$-spheres and connected sums of these. The very first examples of irregular
Sasaki-Einstein metric on $\mathbf{S}^{2}\times \mathbf{S}^{3}$ was constructed
in \cite{gmsw}. We refer to \cite{bg}, \cite{sp} and references therein.

On the other hand, the class of simply connected, closed, oriented, smooth,
$5$-manifolds is classifiable under diffeomorphism due to Smale-Barden
(\cite{s}, \cite{b}). Then it is our goal to focus on the existence of
Sasaki-Einstein metrics or Sasaki-Ricci solitons in a compact quasi-regular
Sasakian manifolds of dimension five and seven.\emph{\ }More precisely, by the
first structure theorem (cf. section $2$), any Hodge orbifold gives rise to a
Sasakian manifold, one expects that existence problems of Sasaki-Einstein
metrics as well as Sasaki-Ricci solitons should be intimately related to
existence problems on such corresponding Hodge orbifolds. However, people has
seen some success in finding K\"{a}hler-Einstein by the K\"{a}hler-Ricci flow
on K\"{a}hler manifolds.

Along this spirits, in this paper we will focus on the following Sasaki-Ricci
flow
\begin{equation}%
\begin{array}
[c]{c}%
\frac{\partial}{\partial t}\omega(t)=\omega(t)-\mathrm{Ric}_{\omega(t)}%
^{T},\text{ }\omega(0)=\omega_{0}%
\end{array}
\label{2022}%
\end{equation}
which is introduced by Smoczyk--Wang--Zhang (\cite{swz}) to study the
existence of Sasaki $\eta$-Einstein metrics on Sasakian manifolds. They showed
that the flow has the longtime solution and asymptotic converges to a Sasaki
$\eta$-Einstein metric when the basic first Chern class is negative
($c_{1}^{B}(M)<0$) or null ($c_{1}^{B}(M)=0$). It is wild open when a compact
Sasakian $(2n+1)$-manifold is transverse Fano ($c_{1}^{B}(M)>0).$ In the paper
of \cite{cj}, Collins and Jacob proved that the Sasaki-Ricci flow converges
exponentially fast to a Sasaki-Einstein metric if one exists, provided the
automorphism group of the transverse holomorphic structure is trivial. In
general, by comparing the K\"{a}hler-Ricci flow on log Fano varieties as in
(\cite{bbegz}), it is hard to deal with because the space of leaves of the
characteristic foliation is a polarized, normal projective variety which
endowed with the orbifold structure due to (\ref{41}).

In section $3$, we first start to consider the most simple case for $n=1$, in
particular, the space $Z$ of leaves will be $(\mathbf{S}^{2},g)$ with branch
divisors $k$ marked points as in Theorem \ref{T22}. \ More precisely, let
$(M,\xi,\eta,g,\Phi)$ be a compact transverse Fano Sasakian $3$-manifold. It
follows from \cite{gei} that any Sasakian $3$-manifold $M$ is either
canonical, anticanonical or null. $M$ is up to finite quotient a regular
Sasakian $3$-manifold, i.e., a circle bundle over a Riemann surface of
positive genus. In the positive case, $M$ is finite covered by $\mathbf{S}%
^{3}$ and its Sasakian structure is a deformation of a standard Sasakian
structure on $\mathbf{S}^{3}$.

Now first we will focus on the Sasaki-Ricci flow (\ref{1}) on the transverse
Fano three-sphere $\mathbf{S}^{3}.$ In fact, it is known that

\begin{proposition}
\label{P21-2} (\cite{wz}, \cite{he}) For any initial Sasakian structure on
$(\mathbf{S}^{3},\xi,g_{0})$ with the positive transverse scalar curvature,
the Sasaki-Ricci flow (\ref{1}) converges exponentially to a gradient
shrinking Sasaki-Ricci soliton%
\[
\nabla_{i}\nabla_{j}f-\frac{1}{2}(\triangle_{B}f)g_{ij}=0.
\]
Here $f$ is the basic function defined by $\triangle_{B}f=R^{T}-r.$ Moreover,
the Sasaki-Ricci soliton metric is a simple Sasaki metric which can be
deformed to the round metric on $\mathbf{S}^{3}$ through a simple deformation.
\end{proposition}

Our first goal is to remove the positive assumption of the initial transverse
scalar curvature, we prove that the transverse scalar curvature $R^{T}$
becomes positive in finite time under the Sasaki-Ricci flow. Then, by applying
Proposition \ref{P21-2}, we have

\begin{theorem}
\label{T21} For any initial Sasakian structure on the transverse Fano
three-sphere $(\mathbf{S}^{3},\xi,g_{0})$, the Sasaki-Ricci flow converges
exponentially to a gradient shrinking Sasaki-Ricci soliton.
\end{theorem}

As in the paper of \cite{h1}, there are no soliton solutions other than those
of constant curvature on a compact surface. However, bad orbifold surfaces do
not admit metrics of constant curvature (\cite{wu}). Then it is very
interested to know whether the Sasaki-Ricci soliton is the trivial one.

It follows from Proposition \ref{P22} and Proposition \ref{P23} that we have
the following classification of Sasaki-Ricci solitons in a compact
quasi-regular Fano Sasakian three-sphere.

\begin{theorem}
\label{T22} Let $(\mathbf{S}^{3},\xi,g_{0})$ be a compact quasi-regular Fano
Sasakian three-sphere and an $\mathbf{S}^{1}$-orbibundle $\pi:(\mathbf{S}%
^{3},g_{0})\rightarrow(\mathbf{S}^{2},g)$ with $k$ marked points
$\{p_{i}\}_{i=1}^{k}$ of the ramification index $m_{i}$ and a metric $g$ on
$\mathbf{S}^{2}$ with a conical singularity of cone angle $\frac{2\pi}{m_{i}}$
at $p_{i}$. Then
\[
k\leq3
\]
so that

\begin{enumerate}
\item For $k=1:$ We have
\[
\chi(\mathbf{S}^{2},\beta)=1+\frac{1}{m_{1}}%
\]
such that the Sasaki-Ricci soliton is nontrivial and $(\mathbf{S}^{2},\beta) $
is a $Z_{m_{1}}$-teardrop.

\item For $k=2:$ We have
\[
\chi(\mathbf{S}^{2},\beta)=\frac{1}{m_{1}}+\frac{1}{m_{2}}%
\]
such that

\begin{enumerate}
\item if $m_{1}=m_{2}$, then the Sasaki-Ricci soliton is trivial with constant
transverse scalar curvature.

\item if $m_{1}\neq m_{2},$ then the Sasaki-Ricci soliton is nontrivial and
$(\mathbf{S}^{2},\beta)$ is the $Z_{m_{1},m_{2}}$-football.
\end{enumerate}

\item For $k=3:$ We have
\[
\chi(\mathbf{S}^{2},\beta)=\frac{1}{m_{1}}+\frac{1}{m_{2}}+\frac{1}{m_{3}}-1
\]
such that the only possible positive numbers for $(m_{1},m_{2},m_{3})$ are
\[
(m,2,2);(3,3,2);(4,3,2);(5,3,2)
\]
and then (\ref{1aa}) holds. Therefore the Sasaki-Ricci soliton is trivial with
constant transverse scalar curvature.
\end{enumerate}
\end{theorem}

As a consequence of Theorem \ref{T22}, we have

\begin{corollary}
There are only two nontrivial Sasaki-Ricci solitons on a compact quasi-regular
Fano Sasakian three-sphere with its leave space $Z_{m_{1}}$-teardrop and
$Z_{m_{1},m_{2}}$-football, respectively.
\end{corollary}

\begin{remark}
\begin{enumerate}
\item We recapture the classification of all Sasakian structures on compact
quasi-regular Fano Sasakian $\mathbf{S}^{3}$ without using the uniformization
of compact $2$-orbifolds. We refer to Belgun's work (\cite{bel}) for an
another proof.

\item The second structure theorem (\cite{ru}) on compact Sasakian manifolds
of dimension $2n+1$ states that any Sasakian structure $(\xi,\eta,\Phi,g)$ on
$M$ is either quasi-regular or there is a sequence of quasi-regular Sasakian
structures $(M,\xi_{i},\eta_{i},\Phi_{i},g_{i})$ converging in the
compact-open $C^{\infty}$-topology to $(\xi,\eta,\Phi,g).$
\end{enumerate}
\end{remark}

For $n\geq2$ in section $4$, we will assume that $M$ \ is a compact
quasi-regular transverse Fano Sasakian manifold and the space $Z$ of leaves is
well-formed which means its orbifold singular locus and algebro-geometric
singular locus coincide, equivalently $Z$ has no branch divisors.

Let $(M,\eta,\xi,\Phi,g)$ be a compact quasi-regular Sasakian $(2n+1)$%
-manifold and $Z=M/F_{\xi}$ denote the space of leaves of the characteristic
foliation which is well-formed, a normal projective variety with codimension
two orbifold singularities $\Sigma$. Then by the first structure theorem
again, $M$ \ is a principal $S^{1}$-orbibundle ($V$-bundle) over $Z$ which is
also a $Q$-factorial, polarized, normal projective variety such that there is
an orbifold Riemannian submersion$\ $%
\begin{equation}
\pi:(M,g)\rightarrow(Z,\omega) \label{41-1}%
\end{equation}
and
\begin{equation}
K_{M}^{T}=\pi^{\ast}(K_{\emph{Z}}^{orb}). \label{41}%
\end{equation}
If the orbifold structure of the leave space $Z$ is well-formed, then the
orbifold canonical divisor $K_{\emph{Z}}^{orb}$ and canonical divisor $K_{Z}$
are the same and thus
\[
K_{M}^{T}=\pi^{\ast}(\varphi^{\ast}K_{\emph{Z}}).
\]

We shall work on the Sasaki-Ricci flow (\ref{2022}) in a compact quasi-regular
transverse Fano Sasakian manifold $(M,\xi,\eta_{0},\Phi_{0},g_{0},\omega_{0})$
of dimension five and seven. In the paper of \cite{he}, the author proved that
the Sasaki--Ricci flow converges to a gradient Sasaki--Ricci soliton in a
compact transverse Fano Sasakian $(2n+1)$-manifold with the initial Sasaki
metric of nonnegative transverse bisectional curvature. By removing the
curvature assumption, we obtain the existence theorem of the gradient
Shrinking Sasaki-Ricci soliton metric in a compact quasi-regular transverse
Fano Sasakian manifold dimension up to seven. More precisely, first it follows
from Theorem \ref{T61}, Theorem \ref{T63}, and Corollary \ref{C62} that we have

\begin{theorem}
\label{T66} Let $(M,\xi,\eta_{0},g_{0})$ be a compact quasi-regular transverse
Fano Sasakian manifold of dimension up to seven and $(Z_{0}=M/\mathcal{F}%
_{\xi},h_{0},\omega_{h_{0}})$ denote the space of leaves of the characteristic
foliation which is a normal Fano projective K\"{a}hler orbifold surface and
well-formed with codimension two orbifold singularities $\Sigma_{0}$. Then,
under the Sasaki-Ricci flow (\ref{2022}), $(M(t),\xi,\eta(t),g(t))$ converges
to a compact quasi-regular transverse Fano Sasakian orbifold manifold
$(M_{\infty},\xi,\eta_{\infty},g_{\infty})$ with the leave space of orbifold
K\"{a}hler manifold $(Z_{\infty}=M_{\infty}/\mathcal{F}_{\xi},h_{\infty})$
which can have at worst codimension two orbifold singularities $\Sigma
_{\infty}.$ Furthermore, $g^{T}(t_{i})$ converges to a gradient Sasaki-Ricci
soliton orbifold metric $g_{\infty}^{\intercal}$ on $M_{\infty}$ with
$g_{\infty}^{\intercal}=\pi^{\ast}(h_{\infty})$ such that $h_{\infty}$ is the
smooth K\"{a}hler-Ricci soliton metric in the Cheeger-Gromov topology on
$Z_{\infty}\backslash \Sigma_{\infty}.$
\end{theorem}

Furthermore, in the following we will show that the gradient Sasaki-Ricci
soliton orbifold metric is a Sasaki-Einstein metric if $M$ is transverse
$K$-stable (Definition \ref{d61}). This is an old dimensional counterpart of
Yau-Tian-Donaldson conjecture on a compact $K$-stable K\"{a}hler manifold
(\cite{cds1}, \cite{cds2}, \cite{cds3}, \cite{t5}). It can be viewed as a
Sasaki analogue of Tian-Zhang's (\cite{tz}) and Chen-Sun-Wang's result
(\cite{csw}) for the K\"{a}hler--Ricci flow.

\begin{theorem}
\label{T67} Let $(M,\xi,\eta_{0},g_{0})$ be a compact quasi-regular transverse
Fano Sasakian manifold of dimension up to seven and $(Z_{0}=M/\mathcal{F}%
_{\xi},h_{0},\omega_{h_{0}})$ be the space of leaves of the characteristic
foliation which is well-formed with codimension two orbifold singularities
$\Sigma_{0}$. If $M$ is transverse stable, then under the Sasaki-Ricci flow,
$M(t)$ converges to a compact transverse Fano Sasakian manifold $M_{\infty}$
which is isomorphic to $M$ endowed with a smooth Sasaki--Einstein metric.
\end{theorem}

\begin{remark}
\begin{enumerate}
\item Note that by continuity method, Collins and Sz\'{e}kelyhidi (\cite{cz2})
showed that a polarized affine variety admits a Ricci-flat K\"{a}hler cone
metric if and only if it is $K$-stable. In particular, the Sasakian manifold
admits a Sasaki-Einstein metric if and only if its K\"{a}hler cone is $K$-stable.

\item On the other hand, instead of $K$-stability on its K\"{a}hler cone, one
can have the so-called transverse $K$-stability on a compact quasi-regular
transverse Fano Sasakian manifold with the space of leaves of the
characteristic foliation which is well-formed and also a normal Fano
projective K\"{a}hler orbifold.

\item In the upcoming paper (\cite{clw2}), under the conic Sasaki-Ricci flow,
we will prove the conic version of Yau-Tian-Donaldson conjecture on a log
transverse Fano Sasakian manifold in which its leave space $Z_{0}$ is not
well-formed. It means that the orbifold structure $(Z_{0},\Delta)$ has the
codimension one fixed point set of some non-trivial isotropy subgroup. This is
served as a Sasaki analogue of the conic K\"{a}hler-Ricci flow as in
\cite{lz}, etc.
\end{enumerate}
\end{remark}

The proofs are given in this paper, primarily along the lines of the arguments
in \cite{tz}. In section $3,$ following Harnack inequality as in \cite{ben},
the transverse scalar curvature $R^{T}$ becomes positive in finite time under
the Sasaki-Ricci flow. Then, by applying Proposition \ref{P21-2}, we have
Theorem \ref{T21}. Furthermore, \ by results as \cite{wu} and \cite{pssw}, we
have the classification of gradient Shrinking Sasaki-Ricci solitons as in
Theorem \ref{T22}.

In section $4,$ the central issue is to show the $L^{4}$-bound of the
transverse Ricci curvature under the Sasaki-Ricci flow. Then, based on
Perelman's uniform noncollapsing condition and pseudolocality theorem of Ricci
flow and a regularity theory for Sasakian manifolds with integral bounded
transverse Ricci curvature, the limit solution $g_{\infty}^{T}$ is a smooth a
smooth Sasaki-Ricci soliton on $(M_{\infty})_{reg}$ which is a $S^{1}$-bundle
over the regular set $\mathcal{R}$ of $Z_{\infty}$ and the singular set
$\mathcal{S}$ of $Z_{\infty}$ is the codimension two orbifold singularities
(Theorem \ref{T64}). It is served as a Sasaki analogue of the regularity
theory of Cheeger-Colding (\cite{cc2}, \cite{cc3}) and Cheeger-Colding-Tian
(\cite{cct}) for manifolds with bounded Ricci curvature. Finally, as a
consequence of the first structure theorem for Sasakian manifolds and the
partial $C^{0}$-estimate (Theorem \ref{T65}), the Gromov--Hausdorff limit
$Z_{\infty}$ is a variety embedded in some $\mathbf{CP}^{N}$ and the singular
set $\mathcal{S}$ is a normal subvariety ( \cite[Theorem 1.6]{t2}). This will
implies Theorem \ref{T66}.

In section $5,$ let $(M,\xi,\eta_{0},g_{0})$ be a compact quasi-regular
transverse Fano Sasakian manifold of dimension up to seven with the space
$(Z=M/\mathcal{F}_{\xi},h_{0},\omega_{h_{0}})$ of leaves of the characteristic
foliation which is well-formed. In this special case, we can following the
notions as in \cite{t3} and \cite{t5} to define the Sasaki analogue of a
$K$-stable Fano K\"{a}hler manifold to be the so-called transverse $K$-stable
on a Sasakian manifold. Finally, by applying the partial $C^{0}$-estimate to
get a lower bound of the transverse Mabuchi $K$-energy. Furthermore,
$Z_{\infty}$ is a normal Fano projective K\"{a}hler orbifold, the
Sasaki-Futaki invariant can be extended to the generalized Sasaki-Futaki
invariant (\ref{59-1}). Then Theorem \ref{T67} follows easily by the
transverse $K$-stable condition.

\textbf{Acknowledgements.} Part of the project was done during the second
named author visiting to Department of Mathematics and Statistics, Georgetown
University from Oct.-Dec. 2021. He would like to express his gratitude for the
warm hospitality there.

\section{Preliminaries}

In this section, we will recall some preliminaries for Sasakian manifolds with
foliation singularities, a Type II deformation of the Sasakian structure and
the Sasaki-Ricci flow. We refer to \cite{bg}, \cite{fow}, \cite{sp},
\cite{clw}, and references therein for some details.

\subsection{Sasakian Structures and Foliation Singularities}

Let $(M,g,\nabla)$ be a Riemannian $(2n+1)$-manifold. $(M,g)$ is called Sasaki
if \ the cone $(C(M),J,\overline{\omega},\overline{g}):=(\mathbf{%
\mathbb{R}
}^{+}\times M\mathbf{,\ }dr^{2}+r^{2}g)$ is K\"{a}hler with $\overline{\omega
}=\frac{1}{2}i\partial \overline{\partial}r^{2}$ and
\[
\overline{\eta}=\frac{1}{2}\overline{g}(\xi,\cdot)\text{ \  \  \textrm{and}
\  \ }\overline{\xi}=J(r\frac{\partial}{\partial r}).
\]
The function $\frac{1}{2}r^{2}$ is hence a global K\"{a}hler potential for the
cone metric. As $\left[  r=1\right]  =\{1\} \times M\subset C(M)$, we may
define the Reeb vector field $\xi$ on $M$ by
\[
\xi=J(\frac{\partial}{\partial r}).
\]
and the contact $1$-form $\eta$ on $TM$
\[
\eta=g(\xi,\cdot)
\]
Then $\xi$ is the killing vector field with unit length such that $\eta
(\xi)=1\ $and$\  \ d\eta(\xi,X)=0.$ The tensor field of $type(1,1)$, defined
by
\[
\Phi(Y)=\nabla_{Y}\xi
\]
satisfies the condition%
\[
(\nabla_{X}\Phi)(Y)=g(\xi,Y)X-g(X,Y)\xi
\]
for any pair of vector fields $X$ and $Y$ on $M$. Then such a triple
$(\eta,\xi,\Phi)$ is called a Sasakian structure on a Sasakian manifold
$(M,g).$ Note that the Riemannian curvature satisfying the following
\[
R(X,\xi)Y=g(\xi,Y)X-g(X,Y)\xi
\]
for any pair of vector fields $X$ and $Y$ on $M$. In particular, the sectional
curvature of every section containing $\xi$ equals one.

\begin{definition}
(\cite{bg}) Let $(M,\eta,\xi,\Phi,g)$ be a compact Sasakian $(2n+1)$-manifold.
If the orbits of the Reeb vector field $\xi$ are all circles, then integrates
to give an isometric $\mathbf{S}^{1}$-action on $(M,g)$. It is nowhere zero
and the action is locally free. Furthermore, the isotropy group of every point
in $M$ is finite. If the action is free, then the Sasakian structure is said
to be regular. Otherwise, it is quasi-regular. If the orbits of are not all
closed, it is called irregular. In this case, the closure of the one parameter
subgroup of the isometry group of $(M,g)$ is isomorphic to a torus
$\mathbf{T}^{k}$. Then the irregular Sasakian manifold has at least an
$\mathbf{T}^{2}$-isometry.
\end{definition}

The first structure theorem on Sasakian manifolds states that

\begin{proposition}
\label{P21}(\cite{ru}, \cite{sp}, \cite{bg}) Let $(M,\eta,\xi,\Phi,g)$ be a
compact quasi-regular Sasakian manifold of dimension $2n+1$ and $Z$ denote the
space of leaves of the characteristic foliation $\mathcal{F}_{\xi}$ (just as
topological space). Then

\begin{enumerate}
\item $Z$ carries the structure of a Hodge orbifold $\mathcal{Z=}(Z,\Delta)$
with an orbifold K\"{a}hler metric $h$ and K\"{a}hler form $\omega$ which
defines an integral class in $H_{orb}^{2}(Z,\mathbf{Z)}$ in such a way that
$\pi:$ $(M,g,\omega)\rightarrow(Z,h,\omega_{h})$ is an orbifold Riemannian
submersion, and a principal $S^{1}$-orbibundle ($V$-bundle) over $Z.$
Furthermore,it satisfies $\frac{1}{2}d\eta=\pi^{\ast}(\omega_{h}).$The fibers
of $\pi$ are geodesics.

\item $Z$ is also a $Q$-factorial, polarized, normal projective algebraic variety.

\item The orbifold $Z$ is Fano if and only if $Ric_{g}>-2$: In this case $Z$
as a topological space is simply connected; and as an algebraic variety is
uniruled with Kodaira dimension $-\infty$.

\item $(M,\xi,g)$ is Sasaki-Einstein if and only if $(Z,h)$ is
K\"{a}hler-Einstein with scalar curvature $4n(n+1).$

\item If $(M,\eta,\xi,\Phi,g)$ is regular then the orbifold structure is
trivial and $\pi$ is a principal circle bundle over a smooth projective
algebraic variety.

\item As real cohomology classes, there is a relation between the first basic
Chern class and the first orbifold Chern class
\[
c_{1}^{B}(M):=c_{1}(\emph{F}_{\xi})=\pi^{\ast}c_{1}^{orb}(\mathbf{Z}).
\]

\end{enumerate}

Conversely, let $\pi$: $M\rightarrow Z$ be a $\mathbf{S}^{1}$-orbibundle over
a compact Hodge orbifold $(Z,h)$ whose first Chern class is an integral class
defined by $[\omega_{Z}]$, and $\eta$ be a $1$-form with $\frac{1}{2}d\eta
=\pi^{\ast}\omega_{Z}$. Then $(M,\pi^{\ast}h+\eta \otimes \eta)$ is a Sasakian
manifold if all the local uniformizing groups inject into the structure group
$U(1).$
\end{proposition}

On the other hand, the second structure theorem on Sasakian manifolds states that

\begin{proposition}
\label{P1} (\cite{ru}) Let $(M,g)$ be a compact Sasakian manifold of dimension
$2n+1$. Any Sasakian structure $(\xi,\eta,\Phi,g)$ on $M$ is either
quasi-regular or there is a sequence of quasi-regular Sasakian structures
$(M,\xi_{i},\eta_{i},\Phi_{i},g_{i})$ converging in the compact-open
$C^{\infty}$-topology to $(\xi,\eta,\Phi,g).$ In particular, if $M$ admits an
irregular Sasakian structure, it admits many locally free circle actions.
\end{proposition}

We recall that

\begin{definition}
(\cite{bg}) An orbifold complex manifold is a normal, compact, complex space
$Z$ locally given by charts written as quotients of smooth coordinate charts.
That is, $Z$ can be covered by open charts $Z=\cup U_{i}.$ The orbifold charts
on $(Z,U_{i},\varphi_{i})$ is defined by the local uniformizing systems
$(\widetilde{U_{i}},G_{i},\varphi_{i})$ centered at the point $p_{i}$, where
$G_{_{i}}$ is the local uniformizing finite group acting on a smooth complex
space $\widetilde{U_{i}}$ such that $\varphi_{i}:\widetilde{U_{i}}\rightarrow
U_{i}=\widetilde{U_{i}}/G_{_{i}}$ is the biholomorphic map. A point $x$ of
complex orbifold $X$ whose isotropy subgroup $\Gamma_{x}\neq Id$ is called a
singular point. Those points with $\Gamma_{x}=Id$ are called regular points.
The set of singular points is called the orbifold singular locus or orbifold
singular set, and is denoted by $\Sigma^{orb}(Z)$.
\end{definition}

Now let $(M,\eta,\xi,\Phi,g)$ be a compact quasi-regular Sasakian manifold of
dimension $2n+1$. By the first structure theorem, the underlying complex space
$\emph{Z}=(Z\emph{,U}_{i}\emph{)}$ is a normal, orbifold variety with the
algebro-geometric singular set $\Sigma(Z)$. Then $\Sigma(Z)\subset \Sigma
^{orb}(Z)$ and it follows that $\Sigma(Z)=\Sigma^{orb}(Z) $ if and only if
none of the local uniformizing groups of the orbifold $(Z\emph{,U}_{i}%
\emph{)}$ contain a reflection. If some $G_{i}$ contains a reflection, then,
on $(Z\emph{,U}_{i}\emph{),}$ the reflection fixes a hyperplane giving rise to
a ramification divisor on $\widetilde{\emph{U}_{i}} $ and a branch divisor
$\Delta$. More precisely, the branch divisor $\Delta$ of an orbifold
$\mathbf{Z}=(Z\emph{,}\Delta)$ is a $Q$-divisor on $Z$ of the form
\[
\Delta=\sum_{\alpha}(1-\frac{1}{m_{\alpha}})D_{\alpha},
\]
where the sum is taken over all Weil divisors $D_{\alpha}$ that lie in
$\Sigma^{orb}(Z)$, and $m_{\alpha}$ is the $gcd$ of the orders of the local
uniformizing groups taken over all points of $D_{\alpha}$ and is called the
ramification index of $D_{\alpha}.$

The orbifold structure $\mathbf{Z}=(Z\emph{,}\Delta)$ is called well-formed if
the fixed point set of every non-trivial isotropy subgroup has codimension at
least two. Then $Z$ is well-formed if and only if its orbifold singular locus
and algebro-geometric singular locus coincide, equivalently $Z$ has no branch divisors.

\begin{example}
For an instance, the weighted projective $\mathbf{CP}(1;4;6)$ has a branch
divisor $\frac{1}{2}D_{0}=\{z_{0}=0\}.$ But $\mathbf{CP}(1;2;3)$ is a
unramified well-formed\ orbifold with two singular points, $(0;1;0)$ with
local uniformizing group the cyclic group $\mathbf{Z}_{2}$, and $(0;0;1)$ with
local uniformizing group $\mathbf{Z}_{3}$. But both are the same varieties.
\end{example}

Note that the orbifold canonical divisor $K_{\emph{Z}}^{orb}$ and canonical
divisor $K_{Z}$ are related by
\begin{equation}
K_{\emph{Z}}^{orb}=\varphi^{\ast}(K_{Z}+[\Delta]). \label{orbifold}%
\end{equation}
In particular, $K_{\emph{Z}}^{orb}=\varphi^{\ast}K_{Z}$ if and only if there
are no branch divisors.

For all previous discussions with the special case for $n=2$, we have the
following result concerning its foliation cyclic quotient
singularities\textbf{:}

\begin{theorem}
\label{T21-2} (\cite{clw}) Let $(M,\eta,\xi,\Phi,g)$ be a compact
quasi-regular Sasakian $5$-manifold and its leave space $Z$ of the
characteristic foliation be well-formed\textbf{.} Then $Z$ is a $Q$-factorial
normal projective algebraic orbifold surface with isolated singularities of a
finite cyclic quotient of $\mathbf{C}^{2}$. Accordingly, $p\in Z$ is
analytically isomorphic to $p\in Z\simeq(0\in \mathbf{C}^{2})/\mu_{Z_{r}}%
,$where $Z_{r}$ is a cyclic group of order $r$ and its action on such open
affine neighborhood is defined by
\[
\mu_{Z_{r}}:(z_{1},z_{2})\rightarrow(\zeta^{a}z_{1},\zeta^{b}z_{2}),
\]
where $\zeta$ is a primitive $r$-th root of unity. We denote the cyclic
quotient singularity by $\frac{1}{r}(a,b)$ with $(a,r)=1=(b,r)$. In
particular, the action can be rescaled so that every cyclic quotient
singularity corresponds to a $\frac{1}{r}(1,a)$-point with $(r,a)=1$. It is
klt (Kawamata log terminal) singularities.
\end{theorem}

\begin{definition}
\label{D21}

\begin{enumerate}
\item Let $(M,\eta,\xi,\Phi,g)$ be a compact quasi-regular Sasakian
$5$-manifold and its leave space $(Z,\emptyset)$ of the characteristic
foliation be well-formed. Then\textbf{\ }the corresponding singularities in
$(M,\eta,\xi,\Phi,g)$\textbf{\ }is called\textbf{\ }foliation cyclic quotient
singularities of type\textbf{\ }%
\[
\frac{1}{r}(1,a)
\]
at a singular fibre $\mathbf{S}^{1}$ in $M$. The foliation singular set is
discrete, and hence finite.

\item Let $(M,\eta,\xi,\Phi,g)$ be a compact quasi-regular Sasakian
$5$-manifold and its leave space $(Z,\Delta)$ has the codimension one fixed
point set of some non-trivial isotropy subgroup. In this case, the action
\[
\mu_{Z_{r}}:(z_{1},z_{2})\rightarrow(e^{\frac{2\pi a_{1}i}{r_{1}}}%
z_{1},e^{\frac{2\pi a_{2}i}{r_{2}}}z_{2}),
\]
for some positive integers $r_{1,}r_{2}$ whose least common multiple is $r$,
and $a_{i},i=1,2$ are integers coprime to $r_{i},i=1,2$. Then the foliation
singular set contains some $3$-dimensional Sasakian submanifolds of $M.$ More
precisely, the corresponding singularities in $(M,\eta,\xi,\Phi,g)$%
\textbf{\ }is called\textbf{\ }the Hopf\textbf{\ }$\mathbf{S}^{1}$%
\textbf{-}orbibundle over a Riemann surface\textbf{\ }$\Sigma_{h}.$
\end{enumerate}
\end{definition}

\subsection{The Foliated Normal Coordinate}

Let $(M,\eta,\xi,\Phi,g)$ be a compact Sasakian $(2n+1)$-manifold with
$g(\xi,\xi)=1$ and the integral curves of $\xi$ are geodesics. For any point
$p\in M$, we can construct local coordinates in a neighborhood of $p$ which
are simultaneously foliated and Riemann normal coordinates (\cite{gkn}). That
is, we can find Riemann normal coordinates $\{x,z^{1},z^{2},\cdot \cdot
\cdot,z^{n}\}$ on a neighborhood $U$ of $p$, such that $\frac{\partial
}{\partial x}=\xi$ on $U$. Let $\{U_{\alpha}\}_{\alpha \in A}$ be an open
covering of the Sasakian manifold and $\pi_{\alpha}:U_{\alpha}\rightarrow
V_{\alpha}\subset%
\mathbb{C}
^{n\text{ }}$ be submersions such that $\pi_{\alpha}\circ \pi_{\beta}^{-1}%
:\pi_{\beta}(U_{\alpha}\cap U_{\beta})\rightarrow \pi_{\alpha}(U_{\alpha}\cap
U_{\beta})$ is biholomorphic. On each $V_{\alpha},$ there is a canonical
isomorphism $d\pi_{\alpha}:D_{p}\rightarrow T_{\pi_{\alpha}(p)}V_{\alpha}$for
any $p\in U_{\alpha},$ where $D=\ker \xi \subset TM.$ Since $\xi$ generates
isometries, the restriction of the Sasakian metric $g$ to $D$ gives a
well-defined Hermitian metric $g_{\alpha}^{T}$ on $V_{\alpha}.$ This Hermitian
metric in fact is K\"{a}hler. More precisely, let $z^{1},z^{2},\cdot \cdot
\cdot,z^{n}$ be the local holomorphic coordinates on $V_{\alpha}$. We pull
back these to $U_{\alpha}$ and still write the same. Let $x$ be the coordinate
along the leaves with $\xi=\frac{\partial}{\partial x}.$ Then we have the
foliation local coordinate $\{x,z^{1},z^{2},\cdot \cdot \cdot,z^{n}\}$ on
$U_{\alpha}\ $and $(D\otimes%
\mathbb{C}
)$ is spanned by the fields $Z_{j}=\left(  \frac{\partial}{\partial z^{j}%
}+ih_{j}\frac{\partial}{\partial x}\right)  ,\  \  \ j\in \left \{
1,2,...,n\right \}  $ with
\[
\eta=dx-ih_{j}dz^{j}+ih_{\overline{j}}d\overline{z}^{j}%
\]
and its dual frame
\[
\{ \eta,dz^{j},\ j=1,2,\cdot \cdot \cdot,n\}.
\]
Here $h$ is a basic function such that $\frac{\partial h}{\partial x}=0$ and
$h_{j}=\frac{\partial h}{\partial z^{j}},h_{j\overline{l}}=\frac{\partial
^{2}h}{\partial z^{j}\partial \overline{z}^{l}}$ with the foliation normal
coordinate%
\begin{equation}
h_{j}(p)=0,h_{j\overline{l}}(p)=\delta_{j}^{l},dh_{j\overline{l}}(p)=0.
\label{AAA3}%
\end{equation}
Moreover, we have
\[
d\eta(Z_{\alpha},\overline{Z_{\beta}})=d\eta(\frac{\partial}{\partial
z^{\alpha}},\frac{\partial}{\overline{\partial}z^{\beta}}).
\]
Then the K\"{a}hler $2$-form $\omega_{\alpha}^{T}$ of the Hermitian metric
$g_{\alpha}^{T}$ on $V_{\alpha},$ which is the same as the restriction of the
Levi form $d\eta$ to $\widetilde{D_{\alpha}^{n}}$, the slice $\{x=$
\textrm{constant}$\}$ in $U_{\alpha},$ is closed. The collection of K\"{a}hler
metrics $\{g_{\alpha}^{T}\}$ on $\{V_{\alpha}\}$ is so-called a transverse
K\"{a}hler metric. We often refer to $d\eta$ as the K\"{a}hler form of the
transverse K\"{a}hler metric $g^{T}$ in the leaf space $\widetilde{D^{n}}.$

The K\"{a}hler form $d\eta$ on $D$ and the K\"{a}hler metric $g^{T}$ is define
such that $g=g^{T}+\eta \otimes \eta.$ Now in terms of the normal coordinate, we
have%
\[
g^{T}=g_{i\overline{j}}^{T}dz^{i}d\overline{z}^{j}.
\]
Here $g_{i\overline{j}}^{T}=g^{T}(\frac{\partial}{\partial z^{i}}%
,\frac{\partial}{\partial \overline{z}^{j}}).$ The transverse Ricci curvature
$Ric^{T}$ of the Levi-Civita connection $\nabla^{T}$ associated to $g^{T}$ is
defined by $Ric^{T}=Ric+2g^{T}$ and then $R^{T}=R+2n.$ The transverse Ricci
form is defined to be $\rho^{T}=Ric^{T}(\Phi \cdot,\cdot)=-iR_{i\overline{j}%
}^{T}dz^{i}\wedge d\overline{z}^{j}$ with
\[
R_{i\overline{j}}^{T}=-\frac{\partial^{2}}{\partial z^{i}\partial \overline
{z}^{j}}\log \det(g_{\alpha \overline{\beta}}^{T})
\]
and it is a closed basic $(1,1)$-form $\rho^{T}=\rho+2d\eta.$

\subsection{The Sasaki-Ricci Flow and Type II Deformations of Sasakian
Structures}

We recall that a $p$-form $\gamma$ on a Sasakian $(2n+1)$-manifold is called
basic if
\[
i(\xi)\gamma=0\text{ \  \  \textrm{and} \  \ }\mathcal{L}_{\xi}\gamma=0.
\]

Let $\Lambda_{B}^{p}$ be the sheaf of germs of basic $p$-forms and
\ $\Omega_{B}^{p}$ be the set of all global sections of $\Lambda_{B}^{p}$. It
is easy to check that $d\gamma$ is basic if $\gamma$ is basic. Set
$d_{B}=d|_{\Omega_{B}^{p}}.$ Then%
\[
d_{B}:=\partial_{B}+\overline{\partial}_{B}:\Omega_{B}^{p}\rightarrow
\Omega_{B}^{p+1}.
\]
with\ $\partial_{B}:\Lambda_{B}^{p,q}\rightarrow \Lambda_{B}^{p+1,q}$ and
$\overline{\partial}_{B}:\Lambda_{B}^{p,q}\rightarrow \Lambda_{B}^{p,q+1}. $
Moreover%
\[
d_{B}d_{B}^{c}=i\partial_{B}\overline{\partial}_{B}\text{ \  \  \textrm{and}
\  \ }d_{B}^{2}=(d_{B}^{c})^{2}=0
\]
for $d_{B}^{c}:=\frac{i}{2}(\overline{\partial}_{B}-\partial_{B}).$ The basic
Laplacian is defined by
\[
\Delta_{B}:=d_{B}d_{B}^{\ast}+d_{B}^{\ast}d_{B}.
\]
Then we have the basic de Rham complex $(\Omega_{B}^{\ast},d_{B})$ and the
basic Dolbeault complex $(\Omega_{B}^{p,\ast},\overline{\partial}_{B})$ and
its cohomology ring $H_{B}^{\ast}(\mathcal{F}_{\xi})\triangleq H_{B}^{\ast
}(M,\mathbf{R})$ of the foliation $\mathcal{F}_{\xi}$ (\cite{eka}]). Then we
can define the orbifold cohomology of the leaf space $Z=M/U(1)$ to be this
basic cohomology ring
\[
H_{orb}^{\ast}(Z,\mathbf{R})\triangleq H_{B}^{\ast}(F_{\xi})
\]
and the basic first Chern class $c_{1}^{B}(M)$ by $c_{1}^{B}=[\frac{\rho^{T}%
}{2\pi}]_{B}$. And a transverse K\"{a}hler-Einstein metric(or a Sasaki $\eta
$-Einstein metric) means that it satisfies $[\rho^{T}]_{B}=\varkappa \lbrack
d\eta]_{B}$ for $\varkappa=-1,0,1$, up to a $D$-homothetic deformation.

\begin{example}
Let $(M,\eta,\xi,\Phi,g)$ be a compact Sasakian $(2n+1)$-manifold. If $g^{T}$
is a transverse K\"{a}hler metric on $M,$ then $h_{\alpha}=\det \left(
((g_{i\overline{j}}^{\alpha})^{T})^{-1}\right)  $ on $U_{\alpha}$ defines a
basic Hermitian metric on the transverse canonical bundle $K_{M}^{T}$. The
inverse $(K_{M}^{T})^{-1}$ of $K_{M}^{T}$ is sometimes called the transverse
anti-canonical bundle. Its basic first Chern class $c_{1}^{B}((K_{M}^{T}%
)^{-1})$ is called the basic first Chern class of $M$ and often denoted by
$c_{1}^{B}(M).$Then it follows from the previous result that $c_{1}^{B}(M)$
that $c_{1}^{B}(M)=[\frac{\rho_{\omega}^{T}}{2\pi}]_{B}$ for any transverse
K\"{a}hler metric $\omega$ on a Sasakian manifold $M$.
\end{example}

\begin{definition}
Let $(L,h)$ be a basic transverse holomorphic line bundle over a Sasakian
manifold $(M,\eta,\xi,\Phi,g)$ with the basic Hermitian metric $h$. We say
that $L$ is\textbf{\ }very ample if for any ordered basis $\underline
{s}=(s_{0},...,s_{N})$ of $H_{B}^{0}(M,L)$, the map $i_{\underline{s}%
}:M\rightarrow \mathbf{CP}^{N}$ given by%
\[
i_{\underline{s}}(x)=[s_{0}(x),...,s_{N}(x)]
\]
is well-defined and an embedding which is $S^{1}$-equivariant with respect to
the weighted $\mathbf{C}^{\ast}$action in $\mathbf{C}^{N+1}$ as long as not
all the $s_{i}(x)$ vanish. We say that $L$ is ample if there exists a positive
integer $m_{0}$ such that $L^{m}$ is very ample for all $m\geq m_{0}.$
\end{definition}

There is a Sasakian analogue of {Kodaira embedding theorem on a compact
quasi-regular Sasakian $(2n+1)$-manifold {due to }\cite{rt}, \cite{hlm} : }

\begin{proposition}
\label{PCR}Let $(M,\eta,\xi,\Phi,g)$ be a compact quasi-regular Sasakian
$(2n+1)$-manifold and $(L,h)$ be a basic transverse holomorphic line bundle
over $M$ with the basic Hermitian metric $h.$ Then $L$ is ample if and only if
$L$ is positive.
\end{proposition}

Now we consider the {Type II deformations of Sasakian structures }$(M,\eta
,\xi,\Phi,g)$ as followings :

By fixing the $\xi$ and varying $\eta$, define
\[
\widetilde{\eta}=\eta+d_{B}^{c}\varphi,
\]
for $\varphi \in \Omega_{B}^{0}$. Then
\[
d\widetilde{\eta}=d\eta+i\partial_{B}\overline{\partial}_{B}\varphi \text{
\  \ and \  \ }\widetilde{\omega}=\omega+i\partial_{B}\overline{\partial}%
_{B}\varphi.
\]
Hence we have the same transversal holomorphic foliation but with the new
K\"{a}hler structure on the K\"{a}hler cone $C(M)$ and new contact bundle
$\widetilde{D}$ with%
\[
\widetilde{\omega}=\frac{1}{2}dd^{c}\widetilde{r}^{2},\widetilde
{r}=re^{\varphi}.
\]
Since $r\frac{\partial}{\partial r}=\widetilde{r}\frac{\partial}%
{\partial \widetilde{r}}$ and $\xi+ir\frac{\partial}{\partial r}=\xi-iJ(\xi)$
is a holomorphic vector field on $C(M),$ so we have the same holomorphic
structure. Finally, by the $\partial_{B}\overline{\partial}_{B}$-Lemma in the
basic Hodge decomposition, there is a basic function $F:M\rightarrow%
\mathbb{R}
$ such that
\[
\rho^{T}(x,t)-\varkappa d\eta(x,t)=d_{B}d_{B}^{c}F=i\partial_{B}%
\overline{\partial}_{B}F.
\]

Now we focus on finding a new $\eta$-Einstein Sasakian structure
$(M,\xi,\widetilde{\eta},\widetilde{\Phi},\widetilde{g})$ with $\widetilde
{g}^{T}=(g_{i\overline{j}}^{T}+\varphi_{i\overline{j}})dz^{i}d\overline{z}%
^{j}$ such that
\[
\widetilde{\rho}^{T}=\varkappa d\widetilde{\eta}.
\]
Hence $\widetilde{\rho}^{T}-\rho^{T}=\kappa d_{B}d_{B}^{c}\varphi-d_{B}%
d_{B}^{c}F$. It follows that there is a Sasakian analogue of the
Monge-Amp\`{e}re equation for the orbifold version of Calabi-Yau Theorem
\begin{equation}
\frac{\det(g_{\alpha \overline{\beta}}^{T}+\varphi_{\alpha \overline{\beta}}%
)}{\det(g_{\alpha \overline{\beta}}^{T})}=e^{-\kappa \varphi+F}. \label{B}%
\end{equation}

Now we consider the Sasaki-Ricci flow on $M\times \lbrack0,T)$%
\[
\frac{d}{dt}g^{T}(x,t)=-(Ric^{T}(x,t)-\varkappa g^{T}(x,t))
\]
which is equivalent to%
\begin{equation}
\frac{d}{dt}\varphi=\log \det(g_{\alpha \overline{\beta}}^{T}+\varphi
_{\alpha \overline{\beta}})-\log \det(g_{\alpha \overline{\beta}}^{T}%
)+\kappa \varphi-F. \label{C}%
\end{equation}
Note that,for any two Sasakian structures with the fixed Reeb vector field
$\xi,$ we have $Vol(M,g)=Vol(M,g^{\prime})$ and
\[
\widetilde{\omega}^{n}\wedge \eta=i^{n}\det(g_{\alpha \overline{\beta}}%
^{T}+\varphi_{\alpha \overline{\beta}})dz^{1}\wedge d\overline{z}^{1}%
\wedge...\wedge dz^{n}\wedge d\overline{z}^{n}\wedge dx.
\]

As before, for an orbifold Riemannian submersion $\pi:(M,g,\omega
)\rightarrow(Z,h,\omega_{h})$ with $\omega=\pi^{\ast}(\omega_{h}).$ We
consider the projection
\begin{equation}
\Pi:(C(M),\overline{g},J,\overline{\omega})\rightarrow(Z,h,\omega_{h})
\label{2022-b}%
\end{equation}
such that $\Pi|_{(M,g,\omega)}=\pi,$ then we have the relation between the
volume form of the K\"{a}hler cone metric on the metric cone and the volume
form of the Sasaki metric on $M$
\begin{equation}
i_{\frac{\partial}{\partial r}}\overline{\omega}^{n+1}=(\Pi^{\ast}\omega
_{h})^{n}\wedge \eta. \label{2022-a}%
\end{equation}
Here $\overline{\omega}^{n+1}=r^{2n+1}(\Pi^{\ast}\omega_{h})^{n}\wedge
dr\wedge \eta.$

\section{The Sasaki-Ricci Flow on Transverse Fano}

Three-Spheres

In this section, we will consider the Sasaki-Ricci flow on the transverse Fano
three-sphere $\mathbf{S}^{3}.$ For simplicity if there is no confusion we
remove the superscript $T$, since all quantities we are considering are
transverse. Let $(x,z=x^{1}+ix^{2})$ be the foliation normal coordinates and
\[
g_{ij}=d\eta(\frac{\partial}{\partial x^{i}},\Phi \frac{\partial}{\partial
x^{j}}).
\]
Note that $R_{ij}=\frac{1}{2}Rg_{ij}$, so we consider the Sasakian-Ricci flow
as
\begin{equation}
\frac{\partial}{\partial t}g_{ij}=(r-R)g_{ij}, \label{1}%
\end{equation}
where $r$ is the average of the transverse scalar curvature $R$. It's not
difficult to see $r$ is a positive constant.

The object of this section is to remove the positive assumption of the initial
transverse scalar curvature. In fact we prove that the transverse scalar
curvature $R$ becomes positive in finite time under the Sasaki-Ricci flow.
Then, by applying Proposition \ref{P21-2} We can derive the Theorem \ref{T21}.

It is proved (\cite{h1}) that the evolution of the transverse scalar curvature
$R$
\begin{equation}
\frac{\partial}{\partial t}R=\triangle_{B}R+R(R-r). \label{2}%
\end{equation}
Then it is natural to consider the ordinary differential equation for $s=s(t)
$
\begin{equation}
\frac{d}{dt}s=s(s-r),\text{ \  \ }s(0)<\min_{x\in M}R(x,0) \label{2'}%
\end{equation}
which is obtained from the parabolic partial differential equation for $R$
simply by dropping the sub-Laplacian term. Note that $s(t)=\frac{r}{1-ce^{rt}%
}<0$ where $c>1.$

The following Harnack inequality is due to \cite{ben}. For completeness, we
sketch the proof here for the transverse scalar curvature $R$ which is a basic function.

\begin{lemma}
Suppose the flow (\ref{1}) have a solution for $t<T^{\ast}(\leq \infty)$. Then
for any two space-times points $(x,\tau)$ and $(y,T)$ with $0<\tau<T<T^{\ast}%
$, we have
\begin{equation}
R(y,T)-s(T)\geq e^{-\frac{D}{4}-c(T-\tau)}(R(x,\tau)-s(\tau)), \label{hk}%
\end{equation}
where
\[
D=D((x,\tau),(y,T))=\inf_{\gamma}\int_{\tau}^{T}|\frac{d}{dt}\gamma|_{g_{t}%
}^{2}dt.
\]
Here the infimum runs over all piece-wisely smooth curves $\gamma(t)$,
$t\in \lbrack \tau,T]$ with $\gamma(\tau)=x$ and $\gamma(T)=y$.
\end{lemma}

\begin{proof}
Note that
\[
\frac{\partial}{\partial t}(R-s)=\triangle_{B}(R-s)+(R-s)(R-r+s).
\]
Since $R-s$ is positive at $t=0,$ then by maximum principle, it stays positive
for all time. So
\[
L=\log(R-s)
\]
is well defined for all times. It follows from (\ref{2}) and (\ref{2'}) that
\[
\frac{\partial L}{\partial t}=\triangle_{B}L+|\nabla L|_{g_{t}}^{2}+R-r+s.
\]
Let
\[
Q=\frac{\partial L}{\partial t}-|\nabla L|_{g_{t}}^{2}-s=\triangle_{B}L+R-r.
\]
We have
\begin{align}
\frac{\partial Q}{\partial t}  &  =\triangle_{B}Q+2g_{t}(\nabla L,\nabla
Q)+2\left \vert \nabla^{2}L+\frac{1}{2}(R-r)g\right \vert _{g_{t}}%
^{2}\label{2021-2}\\
&  +(r-s)Q+s|\nabla L|^{2}+s(R-r).
\end{align}
It is also known that $sR\geq-C$, where $C$ is a positive constant. Thus we
have
\[
\frac{\partial Q}{\partial t}\geq \triangle_{B}Q+2g_{t}(\nabla L,\nabla
Q)+Q^{2}+(r-s)Q+s|\nabla L|^{2}-C.
\]
For $sL$, we have
\[
\frac{\partial(sL)}{\partial t}=\triangle_{B}(sL)+s|\nabla L|_{g_{t}}%
^{2}+s(R-r+s)+s(s-r)L.
\]
Noting that $L\geq-C-Ct$, we have
\[
\frac{\partial(sL)}{\partial t}\geq \triangle_{B}(sL)+2\left \langle \nabla
L,\nabla(sL)\right \rangle _{g_{t}}-s|\nabla L|^{2}-C.
\]
Set $P=Q+sL$. Then we have
\[
\frac{\partial(P)}{\partial t}\geq \triangle_{B}(P)+2g_{t}(\nabla L,\nabla
P)+Q^{2}+(r-s)Q-C.
\]
Since $sL$ is bounded, we can find a positive constant $C>0$ such that for $t
$ large enough enough
\[
\frac{\partial(P)}{\partial t}\geq \triangle_{B}(P)+2g_{t}(\nabla L,\nabla
P)+\frac{1}{2}(P^{2}-C^{2}).
\]
Applying the maximum principle yields
\[
P\geq C\frac{1+ce^{Ct}}{1-ce^{Ct}},
\]
where $c>1$. So we have for $t$ large enough, we have
\[
Q\geq-3C.
\]
Then we have
\[
\frac{\partial L}{\partial t}-|\nabla L|_{g_{t}}^{2}\geq-3C+s\geq-3C-1.
\]
By the fact that $R$ is a basic function, we have
\[
g_{t}^{M}\left(  \nabla^{M}L,\frac{d}{dt}\gamma \right)  =g_{t}^{M}\left(
\nabla L,\frac{d}{dt}\gamma \right)  =g_{t}\left(  \nabla L,\frac{d}{dt}%
\gamma \right)  .
\]
Hence
\[
\frac{d}{dt}L(t,\gamma(t))=\frac{\partial L}{\partial t}+g_{t}\left(  \nabla
L,\frac{d}{dt}\gamma(t)\right)  \geq \frac{\partial L}{\partial t}-|\nabla
L|_{g_{t}}^{2}-\frac{1}{4}\left \vert \frac{d}{dt}\gamma \right \vert _{g_{t}%
}^{2}.
\]
Taking $\gamma(t)$ to be a path achieving the minima $D$, we have
\begin{align}
L(y,T)-L(x,\tau)  &  =\int_{\tau}^{T}\frac{d}{dt}L(t,\gamma(t))dt\nonumber \\
&  \geq-3C(T-\tau)-\frac{D}{4}.\nonumber
\end{align}
This completes the proof of this theorem.
\end{proof}

Now we are ready to prove Theorem \ref{T21} :

\begin{proof}
It follows from Lemma \ref{L61} that the transverse scalar curvature $R$ and
the diameter of $M$ are bounded. Thus, as in section $8$ of \cite{h1}, the
Harnack inequality (\ref{hk}) show that
\[
R-s>C>0.
\]
Since $s$ approaches zero exponentially, we conclude that $R$ becomes positive
in finite time. Then Theorem \ref{T21} follows from Proposition \ref{P21-2}.
\end{proof}

Adapt the notion as in the paper of \cite{pssw}, we first define

\begin{definition}
A metric $g$ on $\Sigma$ is said to have a conical singularity at $p_{i}$ if
it can be expressed as
\[
g=e^{f}|z|^{\beta_{i}}|dz|^{2}%
\]
near $p_{i}$, with $f(z)$ a bounded function. Here $z_{i}$ is a local
holomorphic coordinate centered at $p_{i}$ and $\beta_{i}\in(0,1)$ is a
constant. The cone angle at $p_{i}$ is $2(1-\beta_{i})\pi.$ In the content of
our current paper, we associate the conical singularity to the divisor denoted
by
\[
\beta=\sum_{i}^{k}\beta_{i}[p_{i}]
\]
in a compact Riemann surface $\Sigma$ and refer to the data $(\Sigma,\beta)$
as a pair.
\end{definition}

The orbifold Euler characteristic formula reads as%

\[
\chi(\Sigma,\beta)=\chi(\Sigma)-\sum_{i}^{k}\beta_{i}.
\]

The equation of constant Ricci curvature on the orbifold surface
$(\Sigma,z_{1},...z_{k})$ becomes%
\begin{equation}
Ric(g)=\frac{1}{2}\chi(\Sigma,\beta)g=\frac{1}{2}Rg \label{2021}%
\end{equation}
on $\Sigma \backslash \{p_{1},...p_{k}\}$ with the volume normalization
$\int_{\Sigma}d\mu=2.$

When $\chi(\Sigma,\beta)\leq0,$ it has been shown (\cite{tro}) that it always
admits a conical metric with constant Ricci curvature and such a metric is
unique up to scaling.{\ }

When $\chi(\Sigma,\beta)>0,$ it holds only when $\Sigma=\mathbf{S}^{2}$ and
\begin{equation}
\sum_{i}^{k}\beta_{i}<2. \label{1b}%
\end{equation}
Note that on a compact surface, there are no Ricci soliton solutions other
than those of constant curvature (\cite{h1}). Every bad orbifolds surface do
not admit metrics of constant curvature and so every soliton solution has
nonconstant curvature :

\begin{proposition}
\label{P22}Let $(\mathbf{S}^{2},\beta)$ be a sphere with $k$ marked points
with $\sum_{i}^{k}\beta_{i}<2$ and $k\leq2.$ Then

\begin{enumerate}
\item For $k=1:$ It is a tear-drop. The equation (\ref{2021}) does not admit a
solution. Instead, one can construct a unique rotationally symmetric compact
shrinking soliton $g$ (\cite{wu}, \cite{bm}).

\item For $k=2:$

\begin{enumerate}
\item If $\beta_{1}=\beta_{2},$ there exists a unique rotationally symmetric
compact shrinking solution of equation (\ref{2021}) (\cite{ben}, \cite{bm}).

\item If $\beta_{1}\neq \beta_{2},$ the equation (\ref{2021}) does not admit a
solution. Instead, one can construct a unique rotationally symmetric compact
shrinking soliton (\cite{wu}, \cite{bm}).
\end{enumerate}
\end{enumerate}
\end{proposition}

In the case of $k\geq3,$ the equation (\ref{2021}) admits a unique compact
shrinking soliton.(\cite{tro}, \cite{lt}) if and only if
\begin{equation}
2\max \beta_{j}<\sum_{i}^{k}\beta_{i}. \label{1aa}%
\end{equation}

In the paper of \cite{pssw}, they consider the conical Ricci flow
\begin{equation}
\frac{\partial}{\partial t}g(t)=(\chi(\mathbf{S}^{2},\beta)-R)g(t),g(0)=g_{0}
\label{1a}%
\end{equation}
on $\mathbf{S}^{2}\backslash \beta$, where $g_{0}$ is a metric of
$g_{0}=e^{u_{0}}(\Pi_{i}(\frac{1+|z|^{2}}{|z-p_{i}|^{2}})^{\beta_{i}}%
g_{FS}).u_{0}\in C^{\infty}(\mathbf{S}^{2}),\int_{\mathbf{S}^{2}}d\mu=2.$

\begin{proposition}
\label{P23}(\cite[Theorem 1.3]{pssw}) Let $(\mathbf{S}^{2},\beta)$ be a sphere
with $k$ marked points with $\sum_{i}^{k}\beta_{i}<2$ and $k\geq3.$ If
$(\mathbf{S}^{2},\beta)$ is stable which is (\ref{1aa}), then the flow
(\ref{1a}) converges in the Gromov-Hausdorff topology in $C^{\infty
}(\mathbf{S}^{2}\backslash \beta)$ to the unique conical constant curvature
metric $g_{\infty}\in c_{1}(\mathbf{S}^{2})$ on $(\mathbf{S}^{2},\beta).$
\end{proposition}

The Proof of Theorem \ref{T22} :

\begin{proof}
Let $(\mathbf{S}^{3},\xi,g_{0})$ be a compact quasi-regular Fano Sasakian
three-sphere. By the first structure theorem, there exists a principal
$\mathbf{S}^{1}$-orbibundle $\pi:\mathbf{S}^{3}\longrightarrow \mathbf{S}^{2}$
with $k$ marked points $\{p_{1},...p_{k}\}$ in $(\mathbf{S}^{2},z_{1}%
,...z_{k}),$ $z_{i}$ is a local holomorphic coordinate centered at $p_{i}$.
\ The orbifold structure on $(\mathbf{S}^{2},z_{1},...z_{k})$ is defined by
the local uniformizing systems $(\widetilde{U_{i}},C_{m_{i}},\varphi_{i})$
centered at the point $z_{i}$, where $C_{m_{i}}$ is the cyclic group of order
$m_{i}$ and $\varphi_{i}:\widetilde{U_{i}}\rightarrow U_{i}=\widetilde{U_{i}%
}/C_{m_{i}}$ is the branched covering map $\varphi_{i}(z)=z^{m_{i}}$. Then we
have the following orbifold first Chern class formula with codimension one
canonical divisors $[p_{i}]$ of the ramification index $m_{i}$
\[
c_{1}^{orb}(\mathbf{S}^{2})=c_{1}(\mathbf{S}^{2})-\sum_{i}^{k}(1-\frac
{1}{m_{i}}).
\]
Note that
\[
c_{1}^{B}(\mathbf{S}^{3})=\pi^{\ast}c_{1}^{orb}(\mathbf{S}^{2})
\]
and
\[
\chi(\mathbf{S}^{2},\beta)=c_{1}^{orb}(\mathbf{S}^{2}).
\]
Thus the orbifold first Chern number is satisfying
\[
\chi(\mathbf{S}^{2},\beta)=\chi(\mathbf{S}^{2})-\sum_{i}^{k}(1-\frac{1}{m_{i}%
})=2-k+\sum_{i}^{k}\frac{1}{m_{i}}.
\]
Now a metric $g$ on $\mathbf{S}^{2}$ has a conical singularity at $p_{i}$ with
cone angle at $p_{i}$ is $2(1-\beta_{i})\pi$ so that $\beta_{i}=1-\frac
{1}{m_{i}}.$

Hence $(\mathbf{S}^{3},\xi,g_{0})$ is a Fano Sasakian $3$-sphere ($c_{1}%
^{B}(\mathbf{S}^{3})>0$) only if $k\leq3$ and
\[
0<\chi(\mathbf{S}^{2},\beta)=2-k+\sum_{i}^{k}\frac{1}{m_{i}}=2-\sum_{i}%
^{k\leq3}\beta_{i}%
\]
which is the inequality (\ref{1b}).

Finally, Theorem \ref{T22} follows easily from Proposition \ref{P22} and
Proposition \ref{P23}.
\end{proof}

\section{The Sasaki-Ricci Flow on Transverse Fano Sasakian Manifolds}

\subsection{$L^{4}$-Bound of the Transverse Ricci Curvature}

In this subsection, we show the $L^{4}$-bound of the transverse Ricci
curvature under the Sasaki-Ricci flow.

\begin{theorem}
\label{T51}Let $(M,\xi,\eta_{0},\Phi_{0},g_{0},\omega_{0})$ be a compact
transverse Fano quasi-regular Sasakian $(2n+1)$-manifold and its the space $Z
$ of leaves of the characteristic foliation be well-formed\textbf{.} Then,
under the Sasaki-Ricci flow (\ref{2022}), there exists a positive constant $C
$ such that
\begin{equation}%
\begin{array}
[c]{c}%
\int_{M}|\mathrm{Ric}_{\omega(t)}^{T}|^{4}\omega(t)^{n}\wedge \eta_{0}\leq C,
\end{array}
\label{40}%
\end{equation}
for all $t\in \lbrack0,\infty).$
\end{theorem}

We follow the line in \cite{tz} and \cite{chlw} to prove this estimate. Note
that the flow (\ref{2022}) can be expressed locally as a parabolic
Monge-Amp\`{e}re equation on a basic K\"{a}hler potential $\varphi$ as in
(\ref{C}):%

\begin{equation}%
\begin{array}
[c]{c}%
\frac{d}{dt}\varphi=\log \det(g_{\alpha \overline{\beta}}^{T}+\varphi
_{\alpha \overline{\beta}})-\log \det(g_{\alpha \overline{\beta}}^{T}%
)+\varphi-u(0).
\end{array}
\label{2022-1}%
\end{equation}
Here $u(0)$ is the transverse Ricci potential of $\eta_{0}$, defined by%
\[
R_{kl}^{T}-g_{kl}^{T}=\partial_{k}\overline{\partial}_{l}u(0)
\]
which we normalize so that $\frac{1}{V}\int_{M}e^{-u(0)}\omega_{0}{}^{n}%
\wedge \eta_{0}=1.$ Let $u(t)$ be the evolving transverse Ricci potential.
Then
\begin{equation}%
\begin{array}
[c]{c}%
\partial_{k}\overline{\partial}_{l}\overset{\cdot}{\varphi}=\frac{\partial
}{\partial t}g_{kl}^{T}=g_{kl}^{T}-R_{kl}^{T}=\partial_{k}\overline{\partial
}_{l}u.
\end{array}
\label{41a}%
\end{equation}
It follows from this equation that $\varphi$ evolves by $\overset{\cdot
}{\varphi}(t)=u(t)+c(t)$, for $c(t)$ depending only on time $t$. Then by using
$c(t)$ to adjust the initial value $\varphi(0)$, we always assume that
\[%
\begin{array}
[c]{c}%
\varphi(0)=c_{0}:=\frac{1}{V}\int_{0}^{\infty}e^{-t}||\nabla^{T}\overset
{\cdot}{\varphi}(t)||_{L^{2}}^{2}dt+\frac{1}{V}\int_{M}u(0)\omega_{0}{}%
^{n}\wedge \eta_{0}.
\end{array}
\]
Since
\[%
\begin{array}
[c]{c}%
\partial_{k}\overline{\partial}_{l}(\frac{\partial u}{\partial t}%
)=\partial_{k}\overline{\partial}_{l}u+\partial_{k}\overline{\partial}%
_{l}\Delta_{B}u,
\end{array}
\]
then we can
\[%
\begin{array}
[c]{c}%
a(t)=\frac{1}{V}\int_{M}ue^{-u}\omega(t)^{n}\wedge \eta_{0}%
\end{array}
\]
such that
\[%
\begin{array}
[c]{c}%
\frac{\partial u}{\partial t}=\Delta_{B}u+u-a.
\end{array}
\]
Now by Jensen's inequality, we have $a(t)\leq0$ and then there exist a uniform
positive constant $C_{1}$ such that (\cite{co1})
\begin{equation}
-C_{1}\leq a(t)\leq0 \label{sss}%
\end{equation}
for all $t\geq0.$ Moreover, it follows from the Poincar\'{e} type inequality,
one can show that $a(t)$ increases along the Sasaki-Ricci flow, so we may
assume
\[%
\begin{array}
[c]{c}%
\underset{t\rightarrow \infty}{\lim}a(t)=a_{\infty}.
\end{array}
\]

It follows from \cite{co1} (also \cite{st}) that

\begin{lemma}
\label{L31} Let $(M^{2n+1},\xi,g_{0})$ be a compact Sasakian manifold and let
$g^{T}(t)$ be the solution of the Sasaki-Ricci flow (\ref{2022}) with the
initial transverse metric $g_{0}^{T}$. Then there exists $C$ depending only on
the initial metric such that%
\[
||u(t)||_{C^{0}}+||\nabla^{T}u(t)||_{C^{0}}+||\Delta_{B}u(t)||_{C^{0}}\leq C
\]
for all $t\geq0.$
\end{lemma}

In order to prove the $L^{4}$ bound of transverse Ricci curvature (\ref{40})
under the normalized Sasaki-Ricci flow, it suffices to show that%
\begin{equation}%
\begin{array}
[c]{c}%
\int_{M}|\bigtriangledown^{T}\overline{\bigtriangledown}^{T}u(t)|^{4}%
\omega(t)^{n}\wedge \eta_{0}\leq C,
\end{array}
\label{42}%
\end{equation}
for all $t\geq0$ and for some constant $C$ independent of $t$. We need the
following Lemmas.

\begin{lemma}
\label{L41}There exists a positive constant $C=C(g_{0}^{T})$ such that%
\begin{equation}%
\begin{array}
[c]{c}%
\int_{M}[|\bigtriangledown^{T}\overline{\bigtriangledown}^{T}u|^{2}%
+|\bigtriangledown^{T}\bigtriangledown^{T}u|^{2}+|\mathrm{Rm}^{T}|^{2}%
]\omega(t)^{n}\wedge \eta_{0}\leq C,
\end{array}
\label{43}%
\end{equation}
for all $t\in \lbrack0,\infty).$
\end{lemma}

\begin{proof}
Applying the integration by parts, we have
\[%
\begin{array}
[c]{c}%
\int_{M}[|\bigtriangledown^{T}\overline{\bigtriangledown}^{T}u|^{2}%
\omega(t)^{n}\wedge \eta_{0}=\int_{M}(\Delta_{B}u)^{2}\omega(t)^{n}\wedge
\eta_{0},
\end{array}
\]
and also%
\[%
\begin{array}
[c]{ll}
& \int_{M}|\bigtriangledown^{T}\bigtriangledown^{T}u|^{2}\omega(t)^{n}%
\wedge \eta_{0}\\
= & \int_{M}[(\Delta_{B}u)^{2}-\left \langle \mathrm{Ric}^{T},\partial
_{B}u\overline{\partial}_{B}u\right \rangle ]\omega(t)^{n}\wedge \eta_{0}\\
= & \int_{M}[(\Delta_{B}u)^{2}-|\bigtriangledown^{T}u|^{2}+\left \langle
\partial_{B}\overline{\partial}_{B}u,\partial_{B}u\overline{\partial}%
_{B}u\right \rangle ]\omega(t)^{n}\wedge \eta_{0}\\
\leq & \int_{M}[(\Delta_{B}u)^{2}+|\bigtriangledown^{T}\overline
{\bigtriangledown}^{T}u|^{2}+|\bigtriangledown^{T}u|^{4}]\omega(t)^{n}%
\wedge \eta_{0}.
\end{array}
\]
Moreover, the $L^{2}$-bound of the transverse Riemannian curvature tensor
follows from uniformly bound of the transverse scalar curvature and the Sasaki
analogue of the Chern-Weil theory as in \cite[Lemma 7.2]{zh2}:%
\[%
\begin{array}
[c]{l}%
\frac{(2\pi)^{2}}{2^{n-2}(n-2)!}\int_{M}[2c_{2}^{B}-\frac{n}{n+1}(c_{1}%
^{B})^{2}]\wedge \omega(t)^{n-2}\wedge \eta_{0}\\
=\frac{1}{2^{n}n!}\int_{M}[|\mathrm{Rm}^{T}|-\frac{2}{n(n+1)}(R^{T})^{2}%
-\frac{(n-1)(n+2)}{n(n+1)}((R^{T})^{2}+(2n(n+1))^{2}]\omega(t)^{n}\wedge
\eta_{0}.
\end{array}
\]

\end{proof}

The following integral inequalities hold by using Lemma \ref{L31} and
integration by parts.

\begin{lemma}
\label{L42}There exists a universal positive constant $C=C(g_{0}^{T})$ such
that%
\begin{equation}%
\begin{array}
[c]{l}%
\int_{M}|\bigtriangledown^{T}\overline{\bigtriangledown}^{T}u|^{4}\omega
^{n}\wedge \eta_{0}\\
\leq C\int_{M}[|\overline{\bigtriangledown}^{T}\bigtriangledown^{T}%
\bigtriangledown^{T}u|^{2}+|\bigtriangledown^{T}\bigtriangledown^{T}%
\overline{\bigtriangledown}^{T}u|^{2}]\omega(t)^{n}\wedge \eta_{0},
\end{array}
\label{44}%
\end{equation}%
\begin{equation}%
\begin{array}
[c]{l}%
\int_{M}|\bigtriangledown^{T}\bigtriangledown^{T}u|^{4}\omega^{n}\wedge
\eta_{0}\\
\leq C\int_{M}[|\overline{\bigtriangledown}^{T}\bigtriangledown^{T}%
\bigtriangledown^{T}u|^{2}+|\bigtriangledown^{T}\bigtriangledown^{T}%
\overline{\bigtriangledown}^{T}u|^{2}+|\bigtriangledown^{T}\bigtriangledown
^{T}\bigtriangledown^{T}u|^{2}]\omega(t)^{n}\wedge \eta_{0},
\end{array}
\label{44a}%
\end{equation}
and%
\begin{equation}%
\begin{array}
[c]{l}%
\int_{M}[|\overline{\bigtriangledown}^{T}\bigtriangledown^{T}\bigtriangledown
^{T}u|^{2}+|\bigtriangledown^{T}\bigtriangledown^{T}\overline{\bigtriangledown
}^{T}u|^{2}+|\bigtriangledown^{T}\bigtriangledown^{T}\bigtriangledown
^{T}u|^{2}]\omega(t)^{n}\wedge \eta_{0}\\
\leq C\int_{M}[|\bigtriangledown^{T}\Delta^{T}u|^{2}+|\bigtriangledown
^{T}\bigtriangledown^{T}u|^{2}+|\mathrm{Rm}^{T}|^{2}]\omega(t)^{n}\wedge
\eta_{0}%
\end{array}
\label{45}%
\end{equation}
for all $t\in \lbrack0,\infty).$
\end{lemma}

Now we can prove a uniform bound of $\int_{M}|\bigtriangledown^{T}%
\overline{\bigtriangledown}^{T}u(t)|^{4}\omega(t)^{n}\wedge \eta_{0}$ under the
Sasaki-Ricci flow.

\begin{proposition}
\label{P41}There exists a positive constant $C=C(g_{0}^{T})$ such that%
\begin{equation}%
\begin{array}
[c]{l}%
\int_{M}[|\overline{\bigtriangledown}^{T}\bigtriangledown^{T}\bigtriangledown
^{T}u|^{2}+|\bigtriangledown^{T}\bigtriangledown^{T}\overline{\bigtriangledown
}^{T}u|^{2}+|\bigtriangledown^{T}\bigtriangledown^{T}\bigtriangledown
^{T}u|^{2}]\omega(t)^{n}\wedge \eta_{0}\\
+\int_{M}[|\bigtriangledown^{T}\overline{\bigtriangledown}^{T}u|^{4}%
+|\bigtriangledown^{T}\bigtriangledown^{T}u|^{4}]\omega(t)^{n}\wedge \eta
_{0}\leq C,
\end{array}
\label{46}%
\end{equation}
for all $t\in \lbrack0,\infty).$
\end{proposition}

\begin{proof}
From Lemma \ref{L42}, it is sufficient to get a uniform bound of $\int
_{M}|\bigtriangledown^{T}\Delta^{T}u|^{2}\omega(t)^{n}\wedge \eta_{0} $ under
the Sasaki-Ricci flow. Since
\[%
\begin{array}
[c]{c}%
(\frac{\partial}{\partial t}-\Delta_{B})\Delta_{B}u=\Delta_{B}%
u-|\bigtriangledown^{T}\overline{\bigtriangledown}^{T}u|^{2},
\end{array}
\]
thus%
\[%
\begin{array}
[c]{c}%
\frac{1}{2}(\frac{\partial}{\partial t}-\Delta_{B})(\Delta_{B}u)^{2}%
=(\Delta_{B}u)^{2}-|\bigtriangledown^{T}\Delta_{B}u|^{2}-\Delta_{B}%
u|\bigtriangledown^{T}\overline{\bigtriangledown}^{T}u|^{2}.
\end{array}
\]
Integrating over the manifold gives%
\[%
\begin{array}
[c]{ll}
& \int_{M}|\bigtriangledown^{T}\Delta_{B}u|^{2}\omega^{n}\wedge \eta_{0}\\
= & \int_{M}[(\Delta_{B}u)^{2}-\Delta_{B}u|\bigtriangledown^{T}\overline
{\bigtriangledown}^{T}u|^{2}-\frac{1}{2}\frac{\partial}{\partial t}(\Delta
_{B}u)^{2}]\omega^{n}\wedge \eta_{0}\\
= & \int_{M}[(\Delta_{B}u)^{2}-\Delta_{B}u|\bigtriangledown^{T}\overline
{\bigtriangledown}^{T}u|^{2}+\frac{1}{2}(\Delta_{B}u)^{3}]\omega^{n}\wedge
\eta_{0}-\frac{1}{2}\frac{d}{dt}\int_{M}(\Delta_{B}u)^{2}\omega^{n}\wedge
\eta_{0}.
\end{array}
\]
Applying the uniform bound of $\Delta_{B}u$ from Lemma \ref{L31} and
(\ref{43}), we obtain%
\begin{equation}%
\begin{array}
[c]{c}%
\int_{t}^{t+1}\int_{M}|\bigtriangledown^{T}\Delta_{B}u|^{2}\omega(s)^{n}%
\wedge \eta_{0}ds\leq C,
\end{array}
\label{47}%
\end{equation}
for all $t\geq0.$ Next we compute%
\[%
\begin{array}
[c]{ll}
& \frac{d}{dt}\int_{M}|\bigtriangledown^{T}\Delta_{B}u|^{2}\omega^{n}%
\wedge \eta_{0}\\
= & \int_{M}[|\bigtriangledown^{T}\Delta_{B}u|^{2}-|\bigtriangledown
^{T}\overline{\bigtriangledown}^{T}\Delta_{B}u|^{2}-|\bigtriangledown
^{T}\bigtriangledown^{T}\Delta_{B}u|^{2}+\Delta_{B}u|\bigtriangledown
^{T}\Delta^{T}u|^{2}]\omega^{n}\wedge \eta_{0}\\
& -\int_{M}\left[  \left \langle \bigtriangledown^{T}|\bigtriangledown
^{T}\overline{\bigtriangledown}^{T}u|^{2},\bigtriangledown^{T}\Delta
_{B}u\right \rangle +\left \langle \overline{\bigtriangledown}^{T}%
|\bigtriangledown^{T}\overline{\bigtriangledown}^{T}u|^{2},\overline
{\bigtriangledown}^{T}\Delta_{B}u\right \rangle \right]  \omega^{n}\wedge
\eta_{0}.
\end{array}
\]
By integration by parts, we obtain%
\[%
\begin{array}
[c]{ll}
& \int_{M}\left \langle \bigtriangledown^{T}|\bigtriangledown^{T}%
\overline{\bigtriangledown}^{T}u|^{2},\bigtriangledown^{T}\Delta
_{B}u\right \rangle \omega^{n}\wedge \eta_{0}\\
\leq & \frac{1}{4}\int_{M}[|\bigtriangledown^{T}\overline{\bigtriangledown
}^{T}\Delta_{B}u|^{2}+|\bigtriangledown^{T}\bigtriangledown^{T}\Delta
_{B}u|^{2}]\omega^{n}\wedge \eta_{0}\\
& +C\int_{M}[|\bigtriangledown^{T}\Delta_{B}u|^{2}+|\bigtriangledown
^{T}\bigtriangledown^{T}\overline{\bigtriangledown}^{T}u|^{2}]\omega^{n}%
\wedge \eta_{0}%
\end{array}
\]
and also%
\[%
\begin{array}
[c]{ll}
& \int_{M}\left \langle \overline{\bigtriangledown}^{T}|\bigtriangledown
^{T}\overline{\bigtriangledown}^{T}u|^{2},\overline{\bigtriangledown}%
^{T}\Delta_{B}u\right \rangle \omega^{n}\wedge \eta_{0}\\
\leq & \frac{1}{4}\int_{M}[|\bigtriangledown^{T}\overline{\bigtriangledown
}^{T}\Delta_{B}u|^{2}+|\bigtriangledown^{T}\bigtriangledown^{T}\Delta
_{B}u|^{2}]\omega^{n}\wedge \eta_{0}\\
& +C\int_{M}[|\bigtriangledown^{T}\Delta_{B}u|^{2}+|\bigtriangledown
^{T}\bigtriangledown^{T}\overline{\bigtriangledown}^{T}u|^{2}]\omega^{n}%
\wedge \eta_{0}.
\end{array}
\]
Therefore, by (\ref{45}) and Lemma \ref{L31},
\[%
\begin{array}
[c]{l}%
\frac{d}{dt}\int_{M}|\bigtriangledown^{T}\Delta_{B}u|^{2}\omega^{n}\wedge
\eta_{0}\\
\leq-\frac{1}{2}\int_{M}[|\bigtriangledown^{T}\overline{\bigtriangledown}%
^{T}\Delta_{B}u|^{2}+|\bigtriangledown^{T}\bigtriangledown^{T}\Delta_{B}%
u|^{2}]\omega^{n}\wedge \eta_{0}+C(1+\int_{M}|\bigtriangledown^{T}\Delta
_{B}u|^{2}\omega^{n}\wedge \eta_{0})\\
\leq C(1+\int_{M}|\bigtriangledown^{T}\Delta_{B}u|^{2}\omega^{n}\wedge \eta
_{0}).
\end{array}
\]
The required uniform bound of $\int_{M}|\bigtriangledown^{T}\Delta_{B}%
u|^{2}\omega^{n}\wedge \eta_{0}$ follows from this and (\ref{47}).
\end{proof}

\subsection{Longtime Behavior of Transverse Ricci potential under the
Sasaki-Ricci Flow}

In this subsection, we will show that the transverse Ricci potentials $u(t)$
which is \ a basic function behaves very well as $t\rightarrow \infty$ under
the Sasaki-Ricci flow. This implies that the limit of Sasaki-Ricci flow should
be the Sasaki-Ricci soliton in the $L^{2}$-topology.

We first define%
\[
\mu(g^{T})=\inf \{ \mathcal{W}^{T}(g^{T},f): \int_{M}e^{-f}\omega^{n}\wedge
\eta_{0}=V\},
\]
where $\mathcal{W}^{T}(g^{T},f)=(2\pi)^{-n}\int_{M}(R^{T}+|\bigtriangledown
^{T}f|^{2}+f)e^{-f}\omega^{n}\wedge \eta_{0}$ as in (\ref{2022-2}) below. Note
that
\[
\mu(g^{T})\leq \frac{1}{V}\int_{M}ue^{-u}\omega^{n}\wedge \eta_{0}.
\]

Now under the Sasaki-Ricci flow, for any backward heat equation (\cite{co1})
\begin{equation}
\frac{\partial f}{\partial t}=-\Delta_{B}f+|\bigtriangledown^{T}f|^{2}%
+\Delta_{B}u, \label{d}%
\end{equation}
we have
\[
\frac{d}{dt}\mathcal{W}^{T}(g^{T},f)=\frac{1}{V}\int_{M}(|\bigtriangledown
^{T}\overline{\bigtriangledown}^{T}(u-f)|^{2}+|\bigtriangledown^{T}%
\bigtriangledown^{T}f|^{2})e^{-f}\omega^{n}\wedge \eta_{0}\geq0
\]
and then
\[
\mu(g_{0}^{T})\leq \mu(g^{T}(t))\leq0
\]
for all $t\geq0.$

\begin{lemma}
\label{L43}Under the Sasaki-Ricci flow, for a smooth basic function $f$
\[
\int_{M}|\bigtriangledown^{T}\bigtriangledown^{T}f|^{2}\omega^{n}\wedge
\eta_{0}\leq C(g_{0}^{T})\int_{M}|\bigtriangledown^{T}\overline
{\bigtriangledown}^{T}f|^{2}\omega^{n}\wedge \eta_{0}%
\]

\end{lemma}

\begin{proof}
We may assume that $\int_{M}fe^{-u}\omega^{n}\wedge \eta_{0}=0.$ It follows
from \cite[Theorem 8.1]{co1}, we have the transverse weighted Poincar\'{e}
inequality%
\[
\frac{1}{V}\int_{M}f^{2}e^{-u}\omega^{n}\wedge \eta_{0}\leq \frac{1}{V}\int
_{M}|\bigtriangledown^{T}f|^{2}e^{-u}\omega^{n}\wedge \eta_{0}+(\frac{1}{V}%
\int_{M}fe^{-u}\omega^{n}\wedge \eta_{0})^{2}%
\]
for all basic function $f\in C_{B}^{\infty}(M;R).$ Thus
\[
\int_{M}f^{2}e^{-u}\omega^{n}\wedge \eta_{0}\leq \int_{M}|\bigtriangledown
^{T}f|^{2}e^{-u}\omega^{n}\wedge \eta_{0}.
\]
It follows from Lemma \ref{L31} that
\[
\int_{M}f^{2}\omega^{n}\wedge \eta_{0}\leq C(g_{0})\int_{M}|\bigtriangledown
^{T}f|^{2}\omega^{n}\wedge \eta_{0}.
\]
Thus
\[%
\begin{array}
[c]{ccl}%
\int_{M}|\bigtriangledown^{T}f|^{2}\omega^{n}\wedge \eta_{0} & = & -\int
_{M}f\Delta_{B}f\omega^{n}\wedge \eta_{0}\\
& \leq & \frac{1}{2C}\int_{M}f^{2}\omega^{n}\wedge \eta_{0}+2C\int_{M}%
(\Delta_{B}f)^{2}\omega^{n}\wedge \eta_{0}\\
& \leq & \frac{1}{2}\int_{M}|\bigtriangledown^{T}f|^{2}\omega^{n}\wedge
\eta_{0}+2C\int_{M}(\Delta_{B}f)^{2}\omega^{n}\wedge \eta_{0}%
\end{array}
\]
and
\[
\int_{M}|\bigtriangledown^{T}f|^{2}\omega^{n}\wedge \eta_{0}\leq C(g_{0}%
)\int_{M}(\Delta_{B}f)^{2}\omega^{n}\wedge \eta_{0}.
\]
On the other hand,
\[%
\begin{array}
[c]{ccl}%
\int_{M}|\bigtriangledown^{T}\bigtriangledown^{T}f|^{2}\omega^{n}\wedge
\eta_{0} & = & \int_{M}((\Delta_{B}f)^{2}-Ric^{T}(\bigtriangledown
^{T}f,\bigtriangledown^{T}f))\omega^{n}\wedge \eta_{0}\\
& = & \int_{M}((\Delta_{B}f)^{2}-|\bigtriangledown^{T}f|^{2})\omega^{n}%
\wedge \eta_{0}+\int_{M}\bigtriangledown_{i}^{T}\overline{\bigtriangledown}%
_{j}^{T}u\overline{\bigtriangledown}_{i}^{T}f\bigtriangledown_{j}^{T}%
f\omega^{n}\wedge \eta_{0}\\
& = & \int_{M}((\Delta_{B}f)^{2}-|\bigtriangledown^{T}f|^{2})\omega^{n}%
\wedge \eta_{0}\\
&  & -\int_{M}\overline{\bigtriangledown}_{j}^{T}u(\Delta_{B}f\bigtriangledown
_{j}^{T}f+\overline{\bigtriangledown}_{i}^{T}f\bigtriangledown_{i}%
^{T}\bigtriangledown_{j}^{T}f)\omega^{n}\wedge \eta_{0}\\
& \leq & \int_{M}((\Delta_{B}f)^{2}-|\bigtriangledown^{T}f|^{2})\omega
^{n}\wedge \eta_{0}\\
&  & +\int_{M}((\Delta_{B}f)^{2}+\frac{1}{2}|\bigtriangledown^{T}%
\bigtriangledown^{T}f|^{2}+|\overline{\bigtriangledown}_{j}^{T}u|^{2}%
|\bigtriangledown^{T}f|^{2})\omega^{n}\wedge \eta_{0}\\
& \leq & \int_{M}(2(\Delta_{B}f)^{2}+\frac{1}{2}|\bigtriangledown
^{T}\bigtriangledown^{T}f|^{2}+C|\bigtriangledown^{T}f|^{2})\omega^{n}%
\wedge \eta_{0}%
\end{array}
\]
and then
\[%
\begin{array}
[c]{ccl}%
\int_{M}|\bigtriangledown^{T}\bigtriangledown^{T}f|^{2}\omega^{n}\wedge
\eta_{0} & \leq & \int_{M}(4(\Delta_{B}f)^{2}+2C|\bigtriangledown^{T}%
f|^{2})\omega^{n}\wedge \eta_{0}\\
& \leq & C\int_{M}(\Delta_{B}f)^{2}\omega^{n}\wedge \eta_{0}\\
& \leq & C(g_{0})\int_{M}|\bigtriangledown^{T}\overline{\bigtriangledown}%
^{T}f|^{2}\omega^{n}\wedge \eta_{0}.
\end{array}
\]

\end{proof}

\begin{theorem}
\label{T42} Under the Sasaki-Ricci flow
\begin{equation}
\int_{0}^{\infty}\int_{M}|\bigtriangledown^{T}\bigtriangledown^{T}u|^{2}%
\omega(t)^{n}\wedge \eta_{0}\wedge dt<\infty. \label{a1}%
\end{equation}
In particular,
\begin{equation}
\int_{M}|\bigtriangledown^{T}\bigtriangledown^{T}u|^{2}\omega(t)^{n}\wedge
\eta_{0}\rightarrow0 \label{a2}%
\end{equation}
as $t\rightarrow \infty.$
\end{theorem}

\begin{proof}
Let $f_{k}$ be a minimizer of $\mu(g^{T}(k))$ with $\int_{M}e^{-f_{k}}%
\omega^{n}\wedge \eta_{0}=V$ and $f_{k}(t)$ be the solution of the backward
heat equation (\ref{d}) on the time interval $[k-1,k]$. Then
\[
\frac{1}{V}\int_{k-1}^{k}\int_{M}(|\bigtriangledown^{T}\bigtriangledown
^{T}f_{k}|^{2}+|\bigtriangledown^{T}\overline{\bigtriangledown}^{T}%
(u-f_{k})|^{2})e^{-f_{k}}\omega^{n}\wedge \eta_{0}\wedge dt\leq \mu
(g^{T}(k))-\mu(g^{T}(k-1)).
\]

Hence by using $\mu(g^{T}(t))\leq0$
\[
\sum_{k=1}^{\infty}\int_{k-1}^{k}\int_{M}(|\bigtriangledown^{T}%
\bigtriangledown^{T}f_{k}|^{2}+|\bigtriangledown^{T}\overline{\bigtriangledown
}^{T}(u-f_{k})|^{2})\omega^{n}\wedge \eta_{0}\wedge dt\leq C(g_{0}^{T}).
\]
On the other hand, by applying the above estimate and Lemma \ref{L43} to
$(u-f_{k}),$ we have%
\[%
\begin{array}
[c]{ccl}%
\int_{0}^{\infty}\int_{M}|\bigtriangledown^{T}\bigtriangledown^{T}u|^{2}%
\omega(t)^{n}\wedge \eta_{0}\wedge dt & \leq & \sum_{k=1}^{\infty}\int
_{M}(2|\bigtriangledown^{T}\bigtriangledown^{T}f_{k}|^{2}+2|\bigtriangledown
^{T}\bigtriangledown^{T}(u-f_{k})|^{2})\omega^{n}\wedge \eta_{0}\wedge dt\\
& \leq & \sum_{k=1}^{\infty}\int_{M}(2|\bigtriangledown^{T}\bigtriangledown
^{T}f_{k}|^{2}+2C|\bigtriangledown^{T}\overline{\bigtriangledown}^{T}%
(u-f_{k})|^{2})\omega^{n}\wedge \eta_{0}\wedge dt\\
& \leq & C(g_{0}^{T}).
\end{array}
\]

This is the estimate (\ref{a1}).

Next, it follows from the straight computation that%
\[
\frac{\partial}{\partial t}|\bigtriangledown^{T}\bigtriangledown^{T}%
u|^{2}=\Delta_{B}|\bigtriangledown^{T}\bigtriangledown^{T}u|^{2}%
-|\bigtriangledown^{T}\bigtriangledown^{T}\bigtriangledown^{T}u|^{2}%
-|\overline{\bigtriangledown}^{T}\bigtriangledown^{T}\bigtriangledown
^{T}u|^{2}-2\mathrm{Rm}^{T}(\bigtriangledown^{T}\bigtriangledown
^{T}u.\overline{\bigtriangledown}^{T}\overline{\bigtriangledown}^{T}u)
\]

and%

\[%
\begin{array}
[c]{l}%
\frac{d}{dt}\int_{M}|\bigtriangledown^{T}\bigtriangledown^{T}u|^{2}%
\omega(t)^{n}\wedge \eta_{0}\\
\leq \int_{M}[(||\nabla u(t)||_{C^{0}}^{2}+||\Delta_{B}u(t)||_{C^{0}%
})|\bigtriangledown^{T}\bigtriangledown^{T}u|^{2}+|\bigtriangledown
^{T}\bigtriangledown^{T}\overline{\bigtriangledown}^{T}u|^{2}+||\nabla
u(t)||_{C^{0}}^{2}|\mathrm{Rm}^{T}|^{2}]\omega(t)^{n}\wedge \eta_{0}.
\end{array}
\]
Then, from Lemma \ref{L31}, Lemma \ref{L41} and Proposition \ref{P41}, we
have
\[
\frac{d}{dt}\int_{M}|\bigtriangledown^{T}\bigtriangledown^{T}u|^{2}%
\omega(t)^{n}\wedge \eta_{0}\leq C(g_{0}^{T}).
\]
Hence
\[
\int_{M}|\bigtriangledown^{T}\bigtriangledown^{T}u|^{2}\omega(t)^{n}\wedge
\eta_{0}\rightarrow0
\]

as $t\rightarrow \infty.$
\end{proof}

Similarly, we have

\begin{theorem}
\label{T41} Under the Sasaki-Ricci flow,
\begin{subequations}
\begin{equation}
\int_{t}^{t+1}\int_{M}|\bigtriangledown^{T}(\Delta_{B}u-|\bigtriangledown
^{T}u|^{2}+u)|^{2}\omega(t)^{n}\wedge \eta_{0}\wedge ds\rightarrow0 \label{b1}%
\end{equation}
as $t\rightarrow \infty$ and then
\end{subequations}
\begin{subequations}
\begin{equation}
\int_{M}(\Delta_{B}u-|\bigtriangledown^{T}u|^{2}+u-a)^{2}\omega(t)^{n}%
\wedge \eta_{0}\rightarrow0 \label{b2}%
\end{equation}
as $t\rightarrow \infty.$
\end{subequations}
\end{theorem}

\begin{proof}
Note that by the transverse Ricci potential relation
\[
\bigtriangledown_{i}^{T}(\Delta_{B}u-|\bigtriangledown^{T}u|^{2}%
+u)=\overline{\bigtriangledown}_{j}^{T}\bigtriangledown_{i}^{T}%
\bigtriangledown_{j}^{T}u-\bigtriangledown_{i}^{T}\bigtriangledown_{j}%
^{T}u\overline{\bigtriangledown}_{j}^{T}u
\]
and then
\[
|\bigtriangledown^{T}(\Delta_{B}u-|\bigtriangledown^{T}u|^{2}+u)|^{2}%
\leq2(|\bigtriangledown_{i}^{T}\bigtriangledown_{j}^{T}u|^{2}|\overline
{\bigtriangledown}_{j}^{T}u|^{2}+|\overline{\bigtriangledown}_{j}%
^{T}\bigtriangledown_{i}^{T}\bigtriangledown_{j}^{T}u|^{2}).
\]
In order to derive (\ref{b1}), it suffices to prove that
\begin{equation}
\int_{t}^{t+1}\int_{M}|\overline{\bigtriangledown}_{j}^{T}\bigtriangledown
_{i}^{T}\bigtriangledown_{j}^{T}u|^{2}\omega(t)^{n}\wedge \eta_{0}\wedge
ds\rightarrow0. \label{b5}%
\end{equation}

Since the Reeb vector field and the transverse holomorphic structure are both
invariant, all the integrands are only involved with the transverse K\"{a}hler
structure $\omega(t)$ and basic functions $u(t)$. Hence, under the
Sasaki--Ricci flow, when one applies integration by parts, the expressions
involved behave essentially the same as in the K\"{a}hler-Ricci flow. Hence
(\ref{b5}) follows easily from subsection $4.1$ and \cite[Proposition 3.2]{tz} .

Next we denote $\Delta_{B}u-|\bigtriangledown^{T}u|^{2}+u-a=H$ with $\int
_{M}He^{-u}\omega(t)^{n}\wedge \eta_{0}=0$. Then, by weighted Poincar\'{e}
inequality (cf. \ref{b4}) and the uniform bound of $u$, we have
\[
\int_{M}H^{2}\omega(t)^{n}\wedge \eta_{0}\leq C\int_{M}|\bigtriangledown
^{T}H|^{2}\omega(t)^{n}\wedge \eta_{0}%
\]
and then from (\ref{b1})
\[
\int_{t}^{t+1}\int_{M}H^{2}\omega(t)^{n}\wedge \eta_{0}\wedge dt\rightarrow0
\]
as $t\rightarrow \infty.$ Since
\[
\frac{\partial H}{\partial t}=\Delta_{B}H+H-\frac{da}{dt}+|\bigtriangledown
^{T}\bigtriangledown^{T}u|^{2},
\]
it follows from Proposition \ref{P41} that
\[%
\begin{array}
[c]{ccl}%
\frac{d}{dt}\int_{M}H^{2}\omega(t)^{n}\wedge \eta_{0} & = & \int_{M}%
2H(\Delta_{B}H+H-\frac{da}{dt}+|\bigtriangledown^{T}\bigtriangledown^{T}%
u|^{2}+\frac{1}{2}H\Delta_{B}u)\omega(t)^{n}\wedge \eta_{0}\\
& \leq & \int_{M}2H(H+|\bigtriangledown^{T}\bigtriangledown^{T}u|^{2}+\frac
{1}{2}H\Delta_{B}u)\omega(t)^{n}\wedge \eta_{0}\\
& \leq & (C+|\Delta_{B}u|^{2})\int_{M}H^{2}\omega(t)^{n}\wedge \eta_{0}%
+\int_{M}|\bigtriangledown^{T}\bigtriangledown^{T}u|^{4}\omega(t)^{n}%
\wedge \eta_{0}\\
& \leq & C(1+\int_{M}H^{2}\omega(t)^{n}\wedge \eta_{0}).
\end{array}
\]
Thus
\[
\frac{d}{dt}\int_{M}H^{2}\omega(t)^{n}\wedge \eta_{0}\leq C(1+\int_{M}%
H^{2}\omega(t)^{n}\wedge \eta_{0})
\]
and
\[
\int_{M}(\Delta_{B}u-|\bigtriangledown^{T}u|^{2}+u-a)^{2}\omega(t)^{n}%
\wedge \eta_{0}\rightarrow0
\]
as $t\rightarrow \infty.$
\end{proof}

\subsection{Regularity of the Limit Space and Its Smooth Convergence}

Firstly, we will apply our previous results plus Cheeger-Colding-Tian
structure theory for K\"{a}hler orbifolds (\cite{cct}, \cite{tz}) to study the
structure of desired limit spaces. Since $(M,\eta,\xi,\Phi,g,\omega)$ is a
compact quasi-regular Sasakian manifold. By the first structure theorem on
Sasakian manifolds, $M$ is a principal $S^{1}$-orbibundle ($V$-bundle) over
$Z$ which is also a $Q$-factorial, polarized, normal projective orbifold such
that there is an orbifold Riemannian submersion$\  \pi:(M,g,\omega
)\rightarrow(Z,h,\omega_{h})$ with
\[%
\begin{array}
[c]{c}%
g=g^{T}+\eta \otimes \eta
\end{array}
\]
and
\[%
\begin{array}
[c]{c}%
g^{T}=\pi^{\ast}(h);\text{ \ }\frac{1}{2}d\eta=\pi^{\ast}(\omega_{h}).
\end{array}
\]
The orbit $\xi_{x}$ is compact for any $x\in M,$ we then define the transverse
distance function as
\[%
\begin{array}
[c]{c}%
d^{T}(x,y)\triangleq d_{g}(\xi_{x},\xi_{y}),
\end{array}
\]
where $d$ is the distance function defined by the Sasaki metric $g.$ Then
\[%
\begin{array}
[c]{c}%
d^{T}(x,y)=d_{h}(\pi(x),\pi(y)).
\end{array}
\]

We define a transverse ball $B_{\xi,g}(x,r)$ as follows :
\[%
\begin{array}
[c]{c}%
B_{\xi,g}(x,r)=\left \{  y:d^{T}(x,y)<r\right \}  =\left \{  y:d_{h}(\pi
(x),\pi(y))<r\right \}  .
\end{array}
\]
Note that when $r$ small enough, $B_{\xi,g}(x,r)$ is a trivial $S^{1}$-bundle
over the geodesic ball $B_{h}(\pi(x),r)$.

Based on Perelman's non-collapsing theorem for a transverse ball along the
unnormalizing Sasaki-Ricci flow, it follows that

\begin{lemma}
\label{L61} (\cite[Proposition 7.2]{co1}, \cite[Lemma 6.2]{he}) Let
$(M^{2n+1},\xi,g_{0})$ be a compact Sasakian manifold and let $g^{T}(t)$ be
the solution of the unnormalizing Sasaki-Ricci flow with the initial
transverse metric $g_{0}^{T}$. Then there exists a positive constant $C$ such
that for every $x\in M$, if $|S^{T}|\leq r^{-2}$ on $B_{\xi,g(t)}(x,r)$ for
$r\in(0,r_{0}]$, where $r_{0}$ is a fixed sufficiently small positive number,
then%
\[%
\begin{array}
[c]{c}%
\mathrm{Vol}(B_{\xi,g(t)}(x,r))\geq Cr^{2n}.
\end{array}
\]
Moreover, the transverse scalar curvature $R^{T}$ and transverse diameters
$d_{g^{T}(t)}^{T}$ are uniformly bounded under the Sasaki-Ricci flow. As a
consequence, there is a uniform constant $C$ such that
\[
\mathrm{diam}(M,g(t))\leq C.
\]

\end{lemma}

Now we are ready to study the structure of the limit space(\cite{chlw}) :

\begin{theorem}
\label{T61} Let $(M_{i},\eta_{i},\xi,\Phi_{i},g_{i})$ be a sequence of
quasi-regular Sasakian $(2n+1)$-manifolds with Sasaki metrics $g_{i}=g_{i}%
^{T}+\eta_{i}\otimes \eta_{i}$ such that for basic potentials $\varphi_{i}$
\[%
\begin{array}
[c]{c}%
\eta_{i}=\eta+d_{B}^{C}\varphi_{i}%
\end{array}
\]
and
\[%
\begin{array}
[c]{c}%
d\eta_{i}=d\eta+\sqrt{-1}\partial_{B}\overline{\partial}_{B}\varphi_{i}.
\end{array}
\]
We denote that $(Z_{i},h_{i},J_{i},\omega_{h_{i}})$ are a sequence of
well-formed normal projective orbifolds surfaces which are the corresponding
foliation leave space with respect to $(M_{i},\eta_{i},\xi,\Phi_{i},g_{i}) $
such that
\[%
\begin{array}
[c]{c}%
\omega_{g_{i}^{T}}=\frac{1}{2}d\eta_{i}=\pi^{\ast}(\omega_{h_{i}});\Phi
_{i}=\pi^{\ast}(J_{i})
\end{array}
\]
Suppose that $(M_{i},\eta_{i},\xi,\Phi_{i},g_{i})$ is a compact smooth
transverse Fano Sasakian $(2n+1)$-manifolds satisfying
\[%
\begin{array}
[c]{c}%
\int_{M}|\mathrm{Ric}_{g_{i}^{T}}^{T}|^{p}\omega_{i}^{n}\wedge \eta \leq \Lambda,
\end{array}
\]
and
\[%
\begin{array}
[c]{c}%
\mathrm{Vol}(\emph{B}_{\xi,g_{i}^{T}}(x_{i},1))\geq \nu
\end{array}
\]
for some $p>n,$ $\Lambda>0,\upsilon>0$. Then passing to a subsequence if
necessary, $(M_{i},\Phi_{i},g_{i},x_{i})$ converges in the Cheeger-Gromov
sense to limit length spaces $(M_{\infty},\Phi_{\infty},d_{\infty},x_{\infty
})$ and then $(Z_{i},J_{i},h_{i},\pi(x_{i}))$ converges to $(Z_{\infty
},J_{\infty},h_{\infty},\pi(x_{\infty}))$ such that

\begin{enumerate}
\item for any $r>0$ and $p_{i}\in M_{i}$ with $p_{i}\rightarrow p_{\infty}\in
M_{\infty},$%
\[%
\begin{array}
[c]{c}%
\mathrm{Vol}(B_{h_{i}}(\pi(p_{i}),r))\rightarrow \mathcal{H}^{2n}(B_{h_{\infty
}}(\pi(p_{\infty}),r))
\end{array}
\]
and
\[%
\begin{array}
[c]{c}%
\mathrm{Vol}(B_{\xi,g_{i}^{T}}(p_{i},r))\rightarrow \mathcal{H}^{2n}%
(B_{\xi,g_{\infty}^{T}}(p_{\infty},r)).
\end{array}
\]
Moreover,
\[%
\begin{array}
[c]{c}%
\mathrm{Vol}(B(p_{i},r))\rightarrow \mathcal{H}^{2n+1}(B(p_{\infty},r)),
\end{array}
\]
where $\mathcal{H}^{m}$ denotes the $m$-dimensional Hausdorff measure.

\item $M_{\infty}$ is a $S^{1}$-orbibundle over the normal projective variety
$Z_{\infty}=M_{\infty}/\mathcal{F}_{\xi}.$

\item $Z_{\infty}=\mathcal{R}\cup \mathcal{S}$ such that $\mathcal{S}$ is a
closed singular set of codimension two and $\mathcal{R}$ consists of points
whose tangent cones are $\mathbb{R}^{2n}.$

\item the convergence on the regular part of $M_{\infty}$ which is a $S^{1}%
$-principle bundle over $\mathcal{R}$ in the $(C^{\alpha}\cap L_{B}^{2,p}%
)$-topology for any $0<\alpha<2-\frac{n}{p}$.
\end{enumerate}
\end{theorem}

By the convergence theorem in Theorem \ref{T61}, we have the regularity of the
limit space : We define a family of Sasaki-Ricci flow $g_{i}^{T}(t)$ by%
\[
(M,g_{i}^{T}(t))=(M,g^{T}(t_{i}+t))
\]
for $t\geq-1$ and $t_{i}\rightarrow \infty$ and for $g_{i}^{T}(t)=\pi^{\ast
}(h_{i}(t))$
\[
(Z,h_{i}(t))=(M,h_{i}(t_{i}+t)).
\]
Now for the associated transverse Ricci potential $u_{i}(t)$ as in Lemma
\ref{L31}, we have
\[
||u_{i}(t)||_{C^{0}}+||\nabla^{T}u_{i}(t)||_{C^{0}}+||\Delta_{B}%
u_{i}(t)||_{C^{0}}\leq C.
\]
Moreover, it follows (\ref{a2}) that
\[
\int_{M}|\bigtriangledown^{T}\bigtriangledown^{T}u|^{2}\omega(t)^{n}\wedge
\eta_{0}\rightarrow0
\]
as $i\rightarrow \infty.$ Furthermore, by Theorem \ref{T51} and a convergence
theorem \ref{T61}, passing to a subsequence if necessary, we have at $t=0,$%
\[
(M,g_{i}^{T}(0))\rightarrow(M_{\infty},g_{\infty}^{T},d_{\infty}^{T})
\]
such that
\[
(Z,h_{i}(0))\rightarrow(Z_{\infty},h_{\infty},d_{h_{\infty}})
\]
in the Cheeger-Gromov sense. Moreover,
\begin{equation}
(g_{i}^{T}(0),u_{i}(0))\overset{C^{\alpha}\cap L_{B}^{2,p}}{\rightarrow
}(g_{\infty}^{T},u_{\infty}) \label{c2}%
\end{equation}
on $(M_{\infty})_{reg}$ which is a $S^{1}$-principle bundle over
$\mathcal{R}.$ The convergence of $u_{i}(0)$ follows from the elliptic
regularity (\cite{co1}, \cite{co2}) of
\[
\Delta_{B}u_{i}(0)=n-R^{T}(g_{i}^{T}(0))\in L_{B}^{p}.
\]

\begin{theorem}
\label{T63} Suppose (\ref{c2}) holds, then $g_{\infty}^{T}$ is smooth and
satisfies the Sasaki-Ricci soliton equation%
\begin{equation}
Ric^{T}(g_{\infty}^{T})+Hess^{T}(u_{\infty})=g_{\infty}^{T} \label{c1}%
\end{equation}
on $(M_{\infty})_{reg}$ which is a $S^{1}$-bundle over $\mathcal{R}.$
Moreover, $\Phi_{\infty}$ is smooth and $g_{\infty}^{T}$ is K\"{a}hler with
respect to $\Phi_{\infty}=\pi^{\ast}(J_{\infty}).$
\end{theorem}

\begin{proof}
Since all $g_{\infty}^{T}$ and $u_{\infty}$ are basic, in the local harmonic
coordinate $\{t,x^{1},x^{2},\cdot \cdot \cdot,x^{2n}\}$, the Sasaki-Ricci
soliton equation (\ref{c1}) is equivalent to
\begin{equation}
(g^{T})^{\alpha \beta}\frac{\partial^{2}g_{\gamma \delta}^{T}}{\partial
x^{\beta}\partial x^{\alpha}}=\frac{\partial^{2}u_{\infty}}{\partial
x^{\gamma}\partial x^{\delta}}+Q(g^{T},\partial g^{T})_{\gamma \delta}%
+T(g^{-1},\partial g^{T},\partial u)_{\gamma \delta}-g_{\gamma \delta}^{T}.
\label{c3}%
\end{equation}

By (\ref{a2}), (\ref{c3}) holds in $L_{B}^{2}((M_{\infty})_{reg})$. But
$g_{\infty}^{T}$ and $u_{\infty}$are $L_{B}^{2,p}$, then (\ref{c3}) holds in
$L_{B}^{p}((M_{\infty})_{reg})$ too. On the other hand, by (\ref{b2}) and
(\ref{c1}), we have that
\begin{equation}
g_{\alpha \beta}^{T}\frac{\partial^{2}u_{\infty}}{\partial x^{\beta}\partial
x^{\alpha}}=(g^{T})^{\alpha \beta}\frac{\partial u_{\infty}}{\partial
x^{\alpha}}\frac{\partial u_{\infty}}{\partial x^{\alpha}}-2u_{\infty
}+2a_{\infty} \label{c4}%
\end{equation}
in $L_{B}^{p}((M_{\infty})_{reg}).$

Then a bootstrap argument as in (\cite{pe}) to the elliptic systems (\ref{c3})
and (\ref{c4}) shows that $g_{\infty}^{T}$ and $u_{\infty}$are smooth on
$(M_{\infty})_{reg}$. The elliptic regularity shows that $\Phi_{\infty}$ is
smooth since $\nabla_{g_{\infty}^{T}}^{T}\Phi_{\infty}=0.$
\end{proof}

By applying the argument as in \cite[Theorem 1.2]{tz} (also \cite{chlw}) to
the normal orbifold variety $Z_{\infty}$ which is mainly depended on the
Perelman's pseudolocality theorem (\cite{p1}), we have the smooth convergence
of the Sasaki-Ricci flow on the regular set $(M_{\infty})_{reg}$ which is a
$S^{1}$-principle bundle over $\mathcal{R}$ and $\mathcal{R}$ is the regular
set of $Z_{\infty}$ :

\begin{theorem}
\label{T64}The limit $(M_{\infty},$ $d_{\infty})$ is smooth on the regular set
$(M_{\infty})_{reg}$ which is a $S^{1}$-principle bundle over the regular set
$\mathcal{R}$ of $Z_{\infty}$ and $d_{\infty}^{T}$ is induced by a smooth
Sasaki-Ricci soliton $g_{\infty}^{T}$ and $g^{T}(t_{i})$ converge to
$g_{\infty}^{T}$ in the $C^{\infty}$-topology on $(M_{\infty})_{reg}.$
Moreover, the singular set $\mathcal{S}$ of $Z_{\infty}$ is the codimension
two orbifold singularities.
\end{theorem}

\begin{proof}
Note that Since $g^{T}$ is a transverse K\"{a}hler metric and basic. It is
evolved by the K\"{a}hler--Ricci flow, then it follows from the standard
computations in K\"{a}hler--Ricci flow (\cite{co3}, \cite{mo}) that
\[
\frac{\partial}{\partial t}Rm^{T}=\Delta_{B}Rm^{T}+Rm^{T}\ast Rm^{T}+Rm^{T}.
\]
Now by Perelman's pseudolocality theorem (\cite{p1}, \cite{tz}), there exists
$\varepsilon_{0},\delta_{0},r_{0}$ which depend on $p$ as in the Theorem
\ref{T61} such that for any $(x_{0},t_{0})$, if
\begin{equation}
\mathrm{Vol}(B_{\xi,g_{i}^{T}(t_{0})}(x_{0},r))\geq(1-\varepsilon
_{0})\mathrm{Vol}(B(0,r)) \label{e1}%
\end{equation}
for some $r\leq r_{0},$ where $\mathrm{Vol}(B(0,r)$ denotes the volume of
Euclidean ball of radius $r$ in $\mathbb{R}^{2n},$ then we have the following
curvature estimate%
\begin{equation}
|Rm^{T}|_{g_{i}^{T}}(x,t)\leq \frac{1}{t-t_{0}} \label{e2}%
\end{equation}
for all $x\in B_{\xi,g_{i}^{T}(t)}(x_{0},\varepsilon_{0}r)$ and $t_{0}<t\leq
t_{0}+\varepsilon_{0}^{2}r^{2}$ and the volume estimate
\begin{equation}
\mathrm{Vol}(B_{\xi,g_{i}^{T}(t)}(x_{0},\delta_{0}\sqrt{t-t_{0}}))\geq
(1-\eta)\mathrm{Vol}(B(0,\delta_{0}\sqrt{t-t_{0}})) \label{e3}%
\end{equation}
for $t_{0}<t\leq t_{0}+\varepsilon_{0}^{2}r^{2}$ and $\eta \leq \varepsilon_{0}$
is the constant such that the $C^{\alpha}$ harmonic radius at $x_{0}$ is
bounded below by $\delta_{0}\sqrt{t-t_{0}}).$

As in Theorem \ref{T61}, the metric $g_{i}^{T}(0)$ converges to $g_{\infty
}^{T}$ in the $(C^{\alpha}\cap L_{B}^{2,p})$-topology on $\mathcal{R}$. Now it
is our goal to show that the metric $g_{i}^{T}(0)$ converges smoothly to
$g_{\infty}^{T}$. For $0<r\leq r_{0}$ and $t\geq-1,$ define
\[
\Omega_{r,i,t}:=\{x\in M\text{ }|\text{ (\ref{e1} ) holds on }B_{\xi,g_{i}%
^{T}(t)}(x,t)\}.
\]
Then (\ref{e3}) implies that
\[
\Omega_{r,i,t}\subset \Omega_{\delta_{0}\sqrt{s},i,t+s}%
\]
for $0<s\leq \varepsilon_{0}^{2}r^{2}.$

Let $r_{j}$ to be a decreasing sequence of radii such that $r_{j}\rightarrow0$
and $t_{j}=-\varepsilon_{0}r_{j}$. Then by applying \cite[(3.42)]{tz}, we may
assume that
\[
\Omega_{r_{j},i,t_{j}}\subset \Omega_{r_{j+1},i,t_{j+1}}.
\]
Then by (\ref{e2})%
\begin{equation}
||Rm^{T}||_{g_{i}^{T}(t)}(x,t)\leq \frac{1}{t-t_{j}} \label{e4}%
\end{equation}
for all $(x,t)$ with
\[
d_{g_{i}^{T}(t)}^{T}(x,\Omega_{r_{j},i,t_{j}})\leq \varepsilon_{0}r_{j}%
,t_{j}<t\leq0.
\]
By Shi's derivative estimate (\cite{shi}) to the curvature, there exist a
sequence of constants $C_{k,j,i}$ such that%
\begin{equation}
||(\nabla^{T})^{k}Rm^{T}||_{g_{i}^{T}(0)}(x,t)\leq C_{k,j,i} \label{e5}%
\end{equation}
on $\Omega_{r_{j},i,t_{j}}.$ Then Passing to a subsequence if necessary, one
can find a subsequence \{$i_{j}$\} of \{$j$\} such that
\[
(\Omega_{r_{j},i_{j},t_{j}},g_{i_{j}}^{T}(t_{j}))\overset{C^{\alpha}%
}{\rightarrow}(\overline{\Omega},g_{\overline{\Omega}}^{T})
\]
and
\[
(\Omega_{r_{j},i_{j},t_{j}},g_{i_{j}}^{T}(0))\overset{C^{\infty}}{\rightarrow
}(\Omega,g_{\Omega}^{T}),
\]
where $(\overline{\Omega},g_{\overline{\Omega}}^{T})$ and $(\Omega,g_{\Omega
}^{T})$ are smooth Riemannian manifolds.
\[
(\Omega,g_{\Omega}^{T})\text{ is isometric to }((M_{\infty})_{reg},d_{\infty
}^{T}).
\]
On the other hand, as in Theorem \ref{T61}, we may also have
\[
(M,g_{i_{j}}^{T}(t_{j}))\overset{d_{G,H}^{T}.}{\rightarrow}(\overline
{M}_{\infty},\overline{d}_{\infty}^{T})
\]
with $\overline{Z}_{\infty}=\overline{\mathcal{R}}\cup \overline{\mathcal{S}}.$
Then%
\begin{equation}
(M_{\infty})_{reg}\text{ is the }S^{1}\text{-principle bundle over
}\mathcal{R} \label{f1}%
\end{equation}
and
\begin{equation}
(\overline{M}_{\infty})_{reg}\text{ is the }S^{1}\text{-principle bundle over
}\overline{\mathcal{R}}. \label{f2}%
\end{equation}
Moreover, by the continuity of volume under the Cheeger-Gromov convergence
(\cite[Claim 3.7]{tz}), we have
\begin{equation}
(\overline{\Omega},g_{\overline{\Omega}}^{T})\text{ is isometric to
}((\overline{M}_{\infty})_{reg},\overline{d}_{\infty}^{T}) \label{f3}%
\end{equation}
and then follows from (\ref{e4}) as in \cite[Claim 3.8]{tz} that
\begin{equation}
(\overline{\Omega},g_{\overline{\Omega}}^{T})\text{ is also isometric to
}(\Omega,g_{\Omega}^{T}). \label{f4}%
\end{equation}
Finally (\ref{f1}), (\ref{f2}), (\ref{f3}) and (\ref{f4}) imply the metric
$g_{i}^{T}(0)$ converges smoothly to $g_{\infty}^{T}$ on $\mathcal{R}$.
\end{proof}

For the solution $(M,\omega(t),g^{T}(t))$ of the Sasaki-Ricci flow and the
line bundle $(K_{M}^{T})^{-1},h(t)=\omega^{n}(t))$ with the basic Hermitian
metric $h(t),$ we work on the evolution of the basic transverse holomorphic
line bundle $((K_{M}^{T})^{-m},h^{m}(t))$ for a large integer $m$ such that
$(K_{M}^{T})^{-m}$ is very ample. We consider the basic embedding (Proposition
\ref{PCR}) which is $S^{1}$-equivariant with respect to the weighted
$\mathbf{C}^{\ast}$action in $\mathbf{C}^{N_{m}+1}$
\[
\Psi:M\rightarrow(\mathbf{CP}^{N_{m}},\omega_{FS})
\]
defined by the orthonormal basic transverse holomorphic section $\{ \sigma
_{0},\sigma_{1},...\sigma_{N}\}$ in $H_{B}^{0}(M,(K_{M}^{T})^{-m})$ with
$N_{m}=\dim H_{B}^{0}(M,(K_{M}^{T})^{-m})-1$ with
\[
\int_{M}(\sigma_{i},\sigma_{j})_{h^{m}(t)}\omega^{n}(t)\wedge \eta_{0}%
=\delta_{ij}.
\]

Define
\begin{equation}
\mathcal{F}_{m}(x,t):=\sum_{\alpha=0}^{N_{m}}||\sigma_{\alpha}||_{h^{m}}%
^{2}(x). \label{58}%
\end{equation}
Note that under these notations, the curvature form of the Chern connection
is
\[
Ric(h(t))=m\omega(t).
\]

The following result is a Sasaki analogue of the partial $C^{0}$-estimate
which was obtained in the K\"{a}hler-Ricci flow (\cite[Theorem 5.1]{tz}) and
the proof there does carry over to our Sasaki setting due to the first
structure theorem again on quasi-regular Sasakian manifolds. For completeness,
we give a sketch for the proof here.

\begin{theorem}
\label{T65} Suppose (\ref{c2}) holds, we have
\begin{equation}
\inf_{t_{i}}\inf_{x\in M}\mathcal{F}_{m}(x,t_{i})>0 \label{d11}%
\end{equation}
for a sequence of $m\rightarrow \infty$.
\end{theorem}

\begin{proof}
The main idea is that our basic function theory on quasi-regular Sasaki
manifolds follows immediately from the standard function theory on the
orbifold quotient (Proposition \ref{P21}). In fact, since the leave space of
the characteristic foliation $(Z,h,\omega_{h})$ is a normal projective variety
with codimension two orbifold singularities with $\pi^{\ast}\omega_{h}=\omega
$. Then, by integration over $M$ via (\ref{2022-a}) and (\ref{2022-b}) as in
(\ref{2022A}), this will give the desired estimates back to the space
$(Z,h,\omega_{h})$ which is the same as in K\"{a}hler case.

Here we described briefly two main ingredients for the reader's completeness.
These are the gradient estimate to plurianti-canonical basic sections and
H\"{o}rmander's $L^{2}$-estimate to $\overline{\partial}_{B}$-operator on
basic $(0,1)$-forms. In case of the Ricci curvature bounded, these estimates
are standard as in (\cite{ds}) and (\cite{t1}). In our situation, due to the
lack of the Ricci curvature bounded, the arguments should be modified as
follows :

(i) The uniform bound of the Sobolev constant $C_{S}(g_{0}^{T},n)$ for the
basic function along the Sasaki-Ricci flow was obtained as in \cite[Theorem
1.1]{co2} :
\begin{equation}
(\int_{M}f^{\frac{2(2n+1)}{2n-1}}\omega(t)^{n}\wedge \eta_{0})^{\frac
{2n-1}{2n+1}}\leq C_{S}(g_{0}^{T},n)\int_{M}(||\bigtriangledown^{T}%
f||^{2}+f^{2})\omega(t)^{n}\wedge \eta_{0} \label{b4}%
\end{equation}
for every $f\in W_{B}^{1,2}(M)$. This makes it possible to apply the standard
iteration arguments of Nash-Moser to the proper equations such as
\[%
\begin{array}
[c]{ccl}%
(\Delta_{B}(|\bigtriangledown^{T}\sigma|^{2}) & = & |\bigtriangledown
^{T}\bigtriangledown^{T}\sigma|^{2}+||\overline{\bigtriangledown}%
^{T}\bigtriangledown^{T}\sigma|^{2}-((n+2)m-1)|\bigtriangledown^{T}\sigma
|^{2}-\\
&  & <\bigtriangledown^{T}\overline{\bigtriangledown}^{T}u(\bigtriangledown
^{T}\sigma,\cdot),\bigtriangledown^{T}\sigma>
\end{array}
\]
for any basic transverse holomorphic section $\sigma \in H_{B}^{0}(M,(K_{M}%
^{T})^{-m}).$ Here $u$ is the transverse Ricci potential as before. This is
can be done by the basic Bochner formula. Then we have (\cite[(5.8)]{tz})
\begin{equation}
||\sigma||_{C^{0}}+\sqrt{m}||\bigtriangledown^{T}\sigma||_{C^{0}}\leq
Cm^{\frac{n}{2}}\int_{M}||\sigma||^{2}\omega(t)^{n}\wedge \eta_{0}. \label{d1}%
\end{equation}

(ii) The $L^{2}$-estimate to $\overline{\partial}_{B}$-operator for basic
sections on quasi-regular Sasakian manifolds follows from the standard
function theory on the orbifold quotient : There exists a $m_{0}$ such that
for for any basic transverse holomorphic section $\sigma \in H_{B}^{0}%
(M,K_{M}^{T-m})$ and $m\geq m_{0}$ with
\[
\overline{\partial}_{B}\sigma=0,
\]
one can find a solution $\vartheta$
\[
\overline{\partial}_{B}\vartheta=\sigma
\]
satisfying the property (\cite[(5.9)]{tz})%
\begin{equation}
\int_{M}||\vartheta||^{2}\omega(t)^{n}\wedge \eta_{0}\leq \frac{4}{m}\int
_{M}||\sigma||^{2}\omega(t)^{n}\wedge \eta_{0}. \label{d2}%
\end{equation}
In fact, it suffices to show that the Hodge basic Laplacian
\[
\Delta_{\overline{\partial}_{B}}=\overline{\partial}_{B}\overline{\partial
}_{B}^{\ast}+\overline{\partial}_{B}^{\ast}\overline{\partial}_{B}\geq \frac
{m}{4}%
\]
on $C_{B}^{0}(M,T^{0,1}M\otimes(K_{M}^{T})^{-m})$ for a larger $m$.

Note that when $r$ small enough, the transverse geodesic ball $B_{\xi,g}(x,r)$
is a trivial $S^{1}$-bundle over the geodesic ball $B_{h}(\pi(x),r)$ in
$(Z,h,\omega_{h},\pi(x))$ described as in Theorem \ref{T61} outside the
singular set of codimension $4$. Then the $L^{2}$-estimate to $\overline
{\partial}$-operator for sections on the K\"{a}hler case can be applied for
basic sections on quasi-regular Sasakian manifolds. Indeed, all the integrands
are only involved with the transverse K\"{a}hler structure $\omega(t)$ and
basic sections. Hence, under the Sasaki--Ricci flow, when one applies the
Weitzenb\"{o}ch type formulae and integration by parts, the expressions
involved behave essentially the same as in the K\"{a}hler-Ricci flow.

Finally, Theorem \ref{T65} follows easily from Theorem \ref{T61}, Theorem
\ref{T64}, (\ref{d1}) and (\ref{d2}). We refer to \cite[Theorem 5.1]{tz} and
\cite{t2} for some details.
\end{proof}

As a consequence of the first structure theorem for Sasakian manifolds and
Theorem \ref{T65}, the Gromov--Hausdorff limit $Z_{\infty}$ is a variety
embedded in some $\mathbf{CP}^{N}$ and the singular set $\mathcal{S}$ is a
subvariety (\cite{ds}, \cite[Theorem 1.6]{t2}). Then one can refine the
regularity of Theorem \ref{T64} as following :

\begin{corollary}
\label{C62} Let $(M,\xi,\eta_{0},g_{0})$ be a compact quasi-regular transverse
Fano Sasakian manifold of dimension up to seven and $(Z_{0}=M/\mathcal{F}%
_{\xi},h_{0},\omega_{h_{0}})$ denote the space of leaves of the characteristic
foliation which is a normal projective variety with codimension two orbifold
singularities. Then under the Sasaki-Ricci flow, $M_{\infty}$ is a $S^{1}%
$-orbibundle over $Z_{\infty}$ which is a normal projective variety and the
singular set $\mathcal{S}$ of $Z_{\infty}$ is a codimension two orbifold singularities.
\end{corollary}

Finally, our main theorem \ref{T66} follows easily from Theorem \ref{T61},
Theorem \ref{T63}, and Corollary \ref{C62}.

\section{Transverse $K$-Stability}

In this section, all transverse quantities on Sasakian manifolds such as
transverse Mabuchi $K$-energy, Sasaki-Futaki invariant, transverse Perelman's
$\mathcal{W}$-functional can be viewed as their K\"{a}hler counterparts
restricted on basic functions and transverse K\"{a}hler structure. Then all
the integrands are only involved with the transverse K\"{a}hler structure.
Hence, under the Sasaki--Ricci flow, the Reeb vector field and the transverse
holomorphic structure are both invariant, and the metrics are bundle-like.
Furthermore, when one applies integration by parts, the expressions involved
behave essentially the same as in the K\"{a}hler case.

\subsection{Transverse Mabuchi $K$-energy and Generalized Sasaki-Futaki
Invariant}

We recall the transverse Mabuchi $K$-energy (\cite{cj}) on a compact
transverse Fano Sasakian $(2n+1)$-manifold along any basic transverse
K\"{a}hler potential $\phi_{s}$ with $\phi_{0}=\varphi_{1}$ and $\phi
_{1}=\varphi_{2}$ :%
\begin{equation}
K_{\eta_{0}}(\varphi_{1},\varphi_{2}):=-\frac{1}{V}\int_{0}^{1}%
{\displaystyle \int \limits_{M}}
\overset{\cdot}{\phi}_{s}(R_{\phi_{s}}^{T}-n)\omega_{\phi_{s}}{}^{n}\wedge
\eta_{\phi_{s}}\wedge ds \label{57b}%
\end{equation}
and then also
\[
K_{\eta_{0}}(\varphi_{1},\varphi_{2}):=-\frac{1}{V}\int_{0}^{1}%
{\displaystyle \int \limits_{M}}
\overset{\cdot}{\phi}_{s}(R_{\phi_{s}}^{T}-n)\omega_{\phi_{s}}{}^{n}\wedge
\eta_{0}\wedge ds.
\]

It follows easily from the definition that

\begin{lemma}
\label{L50} (\cite{cj}, \cite{fow})

\begin{enumerate}
\item $K_{\eta_{0}}$ is independent of the path $\phi_{t}$, where
$\overset{\cdot}{\phi}_{s}=\frac{d\phi}{dt}.$ Furthermore it satisfies the
$1$-cocycle condition%
\[
K_{\eta_{0}}(\varphi_{1},\varphi_{2})+K_{\eta_{0}}(\varphi_{2},\varphi
_{3})=K_{\eta_{0}}(\varphi_{1},\varphi_{3})
\]
and
\[
K_{\eta_{0}}(\varphi_{1}+C_{1},\varphi_{2}+C_{2})=K_{\eta_{0}}(\varphi
_{1},\varphi_{2}).
\]

\item For a family of transverse K\"{a}hler potentials $\varphi=\varphi_{t} $
in (\ref{2022-1}), we have%
\begin{equation}
\frac{dK_{\eta_{0}}(\varphi)}{dt}=-%
{\displaystyle \int \limits_{M}}
||\nabla^{T}u(t)||^{2}\omega(t)^{n}\wedge \eta_{0}. \label{31}%
\end{equation}

\end{enumerate}
\end{lemma}

For the Hamiltonian holomorphic vector field $V,$ $d\pi_{\alpha}(V)$ is a
holomorphic vector field on $V_{\alpha}$ and the complex valued Hamiltonian
function $u_{V}:=\sqrt{-1}\eta(V)$ satisfies
\begin{equation}
\overline{\partial}_{B}u_{V}=-\frac{\sqrt{-1}}{2}i_{V}d\eta. \label{59-2}%
\end{equation}

Assume we normalize that $c_{1}^{B}(M)=[\frac{1}{2}d\eta]_{B},$ there exists a
basic function $h_{\omega}$ such that
\[
Ric^{T}(x,t)-\omega(x,t)=i\partial_{B}\overline{\partial}_{B}h_{\omega}.
\]

\begin{lemma}
(\cite{bgs}, \cite{fow}) The Sasaki-Futaki invariant
\begin{equation}
f_{M}(V)=%
{\displaystyle \int \limits_{M}}
V(h_{\omega})\omega^{n}\wedge \eta \label{59}%
\end{equation}
is only depends on the basic cohomology represented by $d\eta$, and not on the
particular transverse K\"{a}hler metric. It is clear that $f_{M}$ vanishes if
$M$ has a Sasaki-Einstein metric in its basic cohomology class. One also have
the following reformulation :
\begin{equation}
f_{M}(V)=-n\int_{M}u_{V}(\mathrm{Ric}_{\omega}^{T}-\omega)\omega^{n-1}%
\wedge \eta=-\int_{M}u_{V}(R_{\omega}^{T}-n)\omega^{n}\wedge \eta. \label{59-1}%
\end{equation}

\end{lemma}

If $(M,\xi,\eta,g,\omega)$ is a transverse Fano quasi-regular Sasakian
manifold and its leave space $Z$ is well-formed and has at least codimension
two fixed point set of every non-trivial isotropy subgroup. Then
$(Z,h,\omega_{h})$ is a normal Fano projective variety with codimension two
orbifold singularities with $\frac{1}{2}d\eta=\pi^{\ast}\omega_{h}=\omega$.
Then, by integration over the $U(1)$-Reeb fibre with the orbifold structure of
$Z$ via (\ref{2022-a}) and (\ref{2022-b}), (\ref{57b}) and (\ref{59})
precisely reduce to the Mabuchi $K$-energy and the Futaki invariant on
K\"{a}hler manifold or orbifold $Z$ up to a proportional constant,
respectively. Following the notions as in \cite{fow} and \cite{dt}, we have

\begin{theorem}
\label{T2022}Let $(M,\xi,\eta,g,\omega)$ be a compact transverse Fano
quasi-regular Sasakian manifold and its leave space $Z$ be the normal Fano
projective K\"{a}hler orbifold and well-formed. The Sasaki-Futaki invariant
can be extended to the generalized Sasaki-Futaki invariant which has the
following reformulation involving $(Z,h,\omega_{h})$ :%
\begin{equation}
f_{M}(V)=f_{Z}(X) \label{59-3}%
\end{equation}
with
\[
f_{Z}(X)=-n%
{\displaystyle \int \limits_{Z}}
\theta_{X}(\mathrm{Ric}_{\omega_{h}}-\omega_{h})\omega_{h}{}^{n-1}=-%
{\displaystyle \int \limits_{Z}}
\theta_{X}(R_{\omega_{h}}-n)\omega_{h}{}^{n}.
\]
Here $f_{Z}(X)$ is the generalized Futaki invariant on $Z$ as in (\ref{59-4})
and the complex valued Hamiltonian function $u_{V}$ on $M$ satisfies
(\ref{59-2}). Moreover, $\pi_{\ast}V=X$ \ is the admissible holomorphic vector
field $X$ which is the restriction of some holomorphic vector field on
$\mathbf{CP}^{N}$ to $Z$ and $\pi^{\ast}\theta_{X}=u_{V}$ satisfies
\[
\overline{\partial}\theta_{X}=-\sqrt{-1}i_{X}\omega_{h}.
\]

\end{theorem}

\begin{proof}
For the Hamiltonian holomorphic vector field $V,$ by definition the complex
valued Hamiltonian function $u_{V}:=\sqrt{-1}\eta(V)$ satisfies
\[
\overline{\partial}_{B}u_{V}=-\frac{\sqrt{-1}}{2}i_{V}d\eta
\]
and $\pi_{\ast}V=X$ is a admissible holomorphic vector field on $Z$. Therefore
one can define the generalized Futaki invariant (\cite{dt}) on a Fano
K\"{a}hler orbifold $(Z,h,\omega_{h})$ by
\begin{equation}
f_{Z}(X)=%
{\displaystyle \int \limits_{Z}}
X(h_{\omega_{h}})\omega_{h}^{n-1} \label{59-4}%
\end{equation}
with
\[
\mathrm{Ric}_{\omega_{h}}-\omega_{h}=i\partial \overline{\partial}h_{\omega
_{h}}.
\]
Moreover, since $\frac{1}{2}d\eta=\pi^{\ast}\omega_{h}=\omega,$ then for a
smooth function $\theta_{X}$ with $\pi^{\ast}\theta_{X}=u_{V},$ the Futaki
invariant was extended to a Fano K\"{a}hler orbifold $Z$ (\cite{dt},
\cite{t5}) as follows :%
\[
f_{Z}(X)=-n%
{\displaystyle \int \limits_{Z}}
\theta_{X}(\mathrm{Ric}_{\omega_{h}}-\omega_{h})\omega_{h}{}^{n-1}%
\]
with
\[
\overline{\partial}\theta_{X}=-\sqrt{-1}i_{X}\omega_{h}.
\]
Furthermore, by (\ref{2022-a}) and notation as above, let $G_{_{i}}$ be the
local uniformizing finite group acting on a smooth complex space
$\widetilde{U_{i}}$ such that the local uniformizing group injects into $U(1)$
and the map%
\[
\varphi_{i}:U(1)\times \widetilde{U_{i}}\rightarrow U_{i}%
\]
be exactly $|G_{_{i}}|$-to-one on the complement of the orbifold locus as in
\cite{cz1} and \cite[section $6$]{chlw}. Then
\begin{equation}%
\begin{array}
[c]{ccl}%
-n%
{\displaystyle \int \limits_{Z}}
\theta_{X}(\mathrm{Ric}_{\omega_{h}}-\omega_{h})\omega_{h}{}^{n-1} & = & -n%
{\displaystyle \sum \limits_{i}}
\frac{1}{|G_{_{i}}|}\int_{\widetilde{U_{i}}}\theta_{X}\varphi_{i}%
(\mathrm{Ric}_{\omega_{h}}-\omega_{h})\omega_{h}{}^{n-1}\\
& = & -n%
{\displaystyle \sum \limits_{i}}
\frac{1}{|G_{_{i}}|}\int_{U(1)\times \widetilde{U_{i}}}\pi^{\ast}\theta_{X}%
\pi^{\ast}\varphi_{i}\pi^{\ast}[(\mathrm{Ric}_{\omega_{h}}-\omega_{h}%
)\omega_{h}{}^{n-1}]\wedge \eta \\
& = & -n%
{\displaystyle \sum \limits_{i}}
\int_{U_{i}}u_{V}\pi^{\ast}\varphi \pi^{\ast}[(\mathrm{Ric}_{\omega_{h}}%
-\omega_{h})\omega_{h}{}^{n-1}]\wedge \eta \\
& = & -n\int_{M}u_{V}\pi^{\ast}[(\mathrm{Ric}_{\omega_{h}}-\omega_{h}%
)\omega_{h}{}^{n-1}]\wedge \eta \\
& = & -n\int_{M}u_{V}[\pi^{\ast}(\mathrm{Ric}_{\omega_{h}}\omega_{h}{}%
^{n-1})-\pi^{\ast}\omega_{h}{}^{n}]\wedge \eta \\
& = & -n\int_{M}u_{V}(\mathrm{Ric}_{\omega}^{T}-\omega)\omega^{n-1}\wedge \eta.
\end{array}
\label{2022A}%
\end{equation}

Then we are done.
\end{proof}

\begin{remark}
Note that$\mathcal{\ }$the set of all Hamiltonian holomorphic vector fields is
a Lie algebra. Furthermore in the paper of \cite{nt}, they proved that if the
transverse scalar curvature $R^{T}$ is constant, then it is reductive which
extending Lichnerowicz-Matsushima theorem in the K\"{a}hler case.
\end{remark}

\subsection{Transverse Perelman's $\mathcal{W}$-functional}

We recall the transverse Perelman's $\mathcal{W}$-functional on a compact
Sasakian $(2n+1)$-manifold :
\begin{equation}
\mathcal{W}^{T}(g^{T},f,\tau)=(4\pi \tau)^{-n}%
{\displaystyle \int \limits_{M}}
(\tau(R^{T}+|\bigtriangledown^{T}f|^{2})+f-2n)e^{-f}\omega^{n}\wedge \eta_{0},
\label{2022-2}%
\end{equation}
for $f\in C_{B}^{\infty}(M;R)$ and $\tau>0$ and define
\[
\lambda^{T}(g^{T},\tau)=\inf \{ \mathcal{W}^{T}(g^{T},f,\tau):f\in
C_{B}^{\infty}(M;R);%
{\displaystyle \int \limits_{M}}
(4\pi \tau)^{-n}e^{-f}\omega^{n}\wedge \eta_{0}=1\}.
\]

\begin{lemma}
\label{L51}(\cite{co1})

\begin{enumerate}
\item
\[
-\infty<\lambda^{T}(g^{T},\tau)\leq C.
\]

\item There exists $f_{\tau}\in C_{B}^{\infty}(M;R)$ so that $\mathcal{W}%
^{T}(g^{T},f_{\tau},\tau)=\lambda^{T}(g^{T},\tau)$. That is,
\[
\lambda^{T}(g^{T},\tau)=\inf \{ \mathcal{W}^{T}(g^{T},f,\tau):f\in W_{B}^{1,2};%
{\displaystyle \int \limits_{M}}
(4\pi \tau)^{-n}e^{-f}\omega^{n}\wedge \eta_{0}=1\}.
\]

\end{enumerate}
\end{lemma}

Note that $\mathcal{W}$-functional can be expressed as
\[%
\begin{array}
[c]{ccl}%
\mathcal{W}^{T}(g^{T},f) & = & \mathcal{W}^{T}(g^{T},f,\frac{1}{2}%
)+(2\pi)^{-n}(2n)V\\
& = & (2\pi)^{-n}%
{\displaystyle \int \limits_{M}}
(\frac{1}{2}(R^{T}+|\bigtriangledown^{T}f|^{2})+f)e^{-f}\omega^{n}\wedge
\eta_{0},
\end{array}
\]
where $(g^{T},f)$ satisfies $%
{\displaystyle \int \limits_{M}}
e^{-f}\omega^{n}\wedge \eta_{0}=V$ and $\tau=\frac{1}{2}.$ Again
\[
\mu^{T}(g^{T})=\inf \{ \mathcal{W}^{T}(g^{T},f):f\in C_{B}^{\infty}(M;R)\text{
\textrm{with} }%
{\displaystyle \int \limits_{M}}
e^{-f}\omega^{n}\wedge \eta_{0}=V\}.
\]
It follows from Lemma \ref{L51} (also \cite{he}) that

\begin{corollary}
\label{C51}

\begin{enumerate}
\item $\lambda^{T}(g^{T})$ can be attained by some $f$ which satisfies the
Euler-Lagrange equation :
\begin{equation}
\Delta_{B}f+f+\frac{1}{2}(R^{T}-||\bigtriangledown^{T}f||^{2})=\mu^{T}(g^{T})
\label{51}%
\end{equation}
and
\begin{equation}
\delta \mu^{T}(g^{T})=-%
{\displaystyle \int \limits_{M}}
<\delta g^{T},Ric^{T}-g^{T}+\bigtriangledown^{T}\overline{\bigtriangledown
}^{T}f>e^{-f}\omega^{n}\wedge \eta_{0}. \label{52}%
\end{equation}

\item $g^{T}$ is a critical point of \ $\mu^{T}(g^{T})$ if and only if $g^{T}$
is a gradient shrinking Sasaki-Ricci soliton%
\begin{equation}
Ric^{T}+\bigtriangledown^{T}\overline{\bigtriangledown}^{T}f=g^{T} \label{53}%
\end{equation}
where $f$ is a minimizer of $\mathcal{W}^{T}(g^{T},\cdot).$
\end{enumerate}
\end{corollary}

It follows from (\ref{52}) that

\begin{corollary}
Let $(M,\xi,\eta_{0},g_{0})$ be a compact quasi-regular transverse Fano
Sasakian $(2n+1)$-manifold and $(Z_{0}=M/\mathcal{F}_{\xi},h_{0},\omega
_{h_{0}})$ denote the space of leaves of the characteristic foliation which is
well-formed normal projective variety with codimension two orbifold
singularities $\Sigma_{0}$. Then, under the Sasaki-Ricci flow,
\begin{equation}
\frac{d}{dt}\mu^{T}(g_{t}^{T})=%
{\displaystyle \int \limits_{M}}
||(Ric_{g_{t}^{T}}^{T}-g_{t}^{T}+\bigtriangledown^{T}\overline
{\bigtriangledown}^{T}f_{t}||_{g_{t}^{T}}^{2}e^{-f_{t}}\omega^{n}(g_{t}%
^{T})\wedge \eta_{0}. \label{54}%
\end{equation}
Here $f_{t}$ are minimizing solutions of (\ref{51}) associated to metrics
$g_{t}^{T}$ and $\sigma_{t}$ is the family of transverse diffeomorphisms of
$M$ generated by the time-dependent vector field $\frac{1}{2}\bigtriangledown
_{g_{t}^{T}}^{T}f_{t}.$
\end{corollary}

Next, we state some a priori estimates for the minimizing solution $f_{t}$ of
(\ref{51}) under the Sasaki-Ricci \ flow. It can be viewed as their K\"{a}hler
counterparts restricted on basic functions and transverse K\"{a}hler
structure, etc. We refer to the details estimates as in \cite[Proposition
4.2]{tzhu2} and \cite{psswe}.

At first, we can improve the estimates as in Lemma \ref{L31} if the transverse
Mabuchi $K$-energy is bounded from below.

\begin{lemma}
If in addition the transverse Mabuchi $K$-energy is bounded from below on the
space of transverse K\"{a}hler potentials as in Lemma \ref{L31}, then we have

\begin{enumerate}
\item
\[
\lim_{t\rightarrow \infty}|u(t)|_{C^{0}(M)}=0,
\]

\item
\[
\lim_{t\rightarrow \infty}|||\bigtriangledown^{T}u(t)||_{C_{B}^{0}(M,g_{t}%
^{T})}=0,
\]

\item
\[
\lim_{t\rightarrow \infty}|\Delta_{B}u(t)|_{C^{0}(M)}=0.
\]

\end{enumerate}
\end{lemma}

As a consequence, we have

\begin{corollary}
\label{C52} If the transverse Mabuchi $K$-energy is bounded from below on the
space of transverse K\"{a}hler potentials as in Lemma \ref{L31}, then there
exists a sequence of $t_{i}\in \lbrack i,i+1]$ such that

\begin{enumerate}
\item
\[
\lim_{t_{i}\rightarrow \infty}|\Delta_{B}f_{t_{i}}|_{L_{B}^{2}(M,g_{t_{I}}%
^{T})}=0,
\]

\item
\[
\lim_{t_{i}\rightarrow \infty}|||\bigtriangledown^{T}f_{t_{i}}||_{L_{B}%
^{2}(M,g_{t_{I}}^{T})}=0,
\]

\item
\[
\lim_{t_{i}\rightarrow \infty}%
{\displaystyle \int \limits_{M}}
f_{t_{i}}e^{-f_{t_{i}}}\omega_{g_{t_{I}}^{T}}^{n}\wedge \eta_{0}=0.
\]
Moreover, we have
\begin{equation}
\lim_{t\rightarrow \infty}\mu^{T}(g_{t}^{T})=(2\pi)^{-n}(2n)V=\sup \{ \mu
^{T}(\overline{g}^{T}):\overline{g}^{T}\in c_{1}^{B}(M)\}. \label{55}%
\end{equation}

\end{enumerate}
\end{corollary}

\subsection{Transverse $K$-Stability}

Following the notions as in \cite{t3} and \cite{t5}, we define the Sasaki
analogue of a $K$-stable Fano K\"{a}hler manifold on a compact quasi-regular
transverse Fano Sasakian $(2n+1)$-manifold $(M,\xi,\eta_{0},g_{0})$ with the
space $(Z=M/\mathcal{F}_{\xi},h_{0},\omega_{h_{0}})$ of leaves of the
characteristic foliation which is well-formed. Then $Z$ is a normal projective
variety with codimension two orbifold singularities (\cite{clw}). In general,
there is a so-called orbifold stability in a a compact quasi-regular Sasakian
manifold in the sense of Ross-Thomas (\cite{rt}). We also refer to
\cite[section $3$]{cz1} for the application.

By applying first structure theory for a quasi-regular Sasakian manifold,
there exists a Riemannian submersion, $\mathbf{S}^{1}$-orbibundle
$\pi:M\rightarrow Z,$ such that
\[
K_{M}^{T}=\pi^{\ast}(K_{\emph{Z}}^{orb})=\pi^{\ast}(\varphi^{\ast}K_{\emph{Z}%
}).
\]
Then by the CR Kodaira embedding theorem (\cite{rt}, \cite{hlm}), there exists
an embedding
\[
\Psi:M\rightarrow(\mathbf{CP}^{N},\omega_{FS})
\]
defined by the basic transverse holomorphic section $\{s_{0},s_{1},...s_{N}\}
$ of $H_{B}^{0}(M,(K_{M}^{T})^{-m})$ which is $\mathbf{S}^{1}$-equivariant
with respect to the weighted $\mathbf{C}^{\ast}$-action in $\mathbf{C}^{N+1} $
with $N=\dim H_{B}^{0}(M,(K_{M}^{T})^{-m})-1$ for a large positive integer
$m.$ In fact, in our situation $Z$ is also normal Fano, there is an embedding
\[
\psi_{|mK_{\emph{Z}}^{-1}|}:Z\rightarrow \mathcal{P(}H^{0}(Z,K_{\emph{Z}}%
{}^{-m})).
\]
Define
\[
\Psi_{|m(K_{M}^{T})^{-1}|}=\psi_{|mK_{\emph{Z}}^{-1}|}\circ \pi
\]
such that
\[
\Psi_{|m(K_{M}^{T})^{-1}|}:M\rightarrow \mathcal{P(}H_{B}^{0}(M,(K_{M}%
^{T})^{-m}).
\]

We define
\[
\mathcal{D}iff^{T}(M)=\{ \sigma \in \mathcal{D}iff(M)\text{ \ }|\text{ }%
\sigma_{\ast}\xi=\xi \text{ \  \  \textrm{and} \ }\sigma^{\ast}g^{T}%
=(\sigma^{\ast}g)^{T}\}
\]
and
\[
SL^{T}(N+1,\mathbf{C})=SL(N+1,\mathbf{C})\cap \mathcal{D}iff^{T}(M).
\]

Any other basis of $H_{B}^{0}(M,(K_{M}^{T})^{-m})$ gives an embedding of the
form $\sigma^{T}\circ \Psi_{|m(K_{M}^{T})^{-1}|}$ with $\sigma^{T}\in
SL^{T}(N+1,\mathbf{C}).$ Now for any subgroup of the weighted $\mathbf{C}%
^{\ast}$action $G_{0}=\{ \sigma^{T}(t)\}_{t\in \mathbf{C}^{\ast}}$ of
$SL^{T}(N+1,\mathbf{C}),$ there is a unique limiting%
\[
M_{\infty}=\lim_{t\rightarrow0}\sigma^{T}(t)(M)\subset \mathbf{CP}^{N}.
\]
As our application of Theorem \ref{T66}, $M_{\infty}$ has its leave space
$Z_{\infty}=M_{\infty}/\mathcal{F}_{\xi}$ which is a normal projective
K\"{a}hler orbifold with at worst codimension two orbifold singularities
$\Sigma_{\infty}.$

Let $V$ be a the Hamiltonian holomorphic vector field whose real part
generates the action by $\sigma^{T}(e^{-s})$. As the previous discussion and
Theorem \ref{T2022}, if $Z_{\infty}$ is normal Fano, there is a generalized
Futaki invariant defined by $f_{Z_{\infty}}(X)$ and then a generalized
Sasaki-Futaki invariant defined by $f_{M_{\infty}}(V)$ as in (\ref{59}),
(\ref{59-3}) and (\ref{59-4}). Thus one can introduce the Sasaki analogue of
the $K$-stable on K\"{a}hler manifolds (\cite{t3}, \cite{t5}, \cite{d}) :

\begin{definition}
\label{d61}Let $(M,\xi,\eta,g,\omega)$ be a compact transverse Fano
quasi-regular Sasakian manifold and its leave space $(Z,h,\omega_{h})$ be a
normal Fano projective K\"{a}hler orbifold and well-formed. We say that $M$ is
transverse $K$-stable with respect to $(K_{M}^{T})^{-m}$ if the generalized
Sasaki-Futaki invariant
\[
\operatorname{Re}f_{M_{\infty}}(V)\geq0\text{ \  \  \textrm{or} \  \ }%
\operatorname{Re}f_{Z_{\infty}}(X)\geq0
\]
for any weighted $\mathbf{C}^{\ast}$-action $G_{0}=\{ \sigma^{T}%
(t)\}_{t\in \mathbf{C}^{\ast}}$ of $SL^{T}(N+1,\mathbf{C})$ with a normal Fano
$Z_{\infty}=M_{\infty}/\mathcal{F}_{\xi}$ and the equality holds if and only
if $M_{\infty}$ is transverse biholomorphic to $M.$ We say that $M$ is
transverse $K$-stable if it is transverse $K$-stable for all large positive
integer $m$.
\end{definition}

With Definition \ref{d61} in mind, we are ready to show that the transverse
Mabuchi $K$-energy is bounded from below under the Sasaki-Ricci flow. This is
served as the Sasaki analogue of the K\"{a}hler-Ricci flow due to Tian-Zhang
(\cite{tz}).

\begin{theorem}
Let $(M,\xi,\eta_{0},g_{0})$ be a compact quasi-regular transverse Fano
Sasakian manifold of dimension up to seven and $(Z_{0}=M/\mathcal{F}_{\xi
},h_{0},\omega_{h_{0}})$ denote the space of leaves of the characteristic
foliation which is a normal projective K\"{a}hler orbifold with codimension
two orbifold singularities $\Sigma_{0}$. If $M$ is transverse $K$-stable, then
the transverse Mabuchi $K$-energy is bounded from below under the Sasaki-Ricci
flow%
\begin{equation}
K_{\eta_{0}}(\omega_{0},\omega_{t})\geq-C(g_{0}). \label{54-b}%
\end{equation}

\end{theorem}

\begin{proof}
First it follows from (\ref{31}) that $K_{\eta_{0}}(\omega_{0},\omega_{t}) $
is non-increasing in $t$. So it suffices to show a uniform lower bound of
\begin{equation}
K_{\eta_{0}}(\omega_{0},\omega_{t_{i}})\geq-C. \label{56}%
\end{equation}
Now if $M$ is transverse $K$-stable, we fix an integer $m>0$ sufficiently
large such that $(L^{T})^{-m}$ is very-ample. Then for any orthonormal basic
basis $\{ \sigma_{t_{i},m,k}\}_{k=0}^{N_{m}}$ in $H_{B}^{0}(M,(K_{M}^{T}%
)^{-m})$ with $N_{m}=\dim H_{B}^{0}(M,(K_{M}^{T})^{-m})-1$ at $t_{i}$ , we can
define the $S^{1}$-equivariant embedding (\cite{chlw}) with respective the
weighted $\mathbf{C}^{\ast}$-action in $\mathbf{C}^{N_{m}+1}$
\[
\Psi_{i}:M\rightarrow(\mathbf{CP}^{N_{m}},\omega_{FS})
\]
with the Bergman metric $\overline{\omega}_{t_{i}}:=\overline{\omega}%
_{i}=\frac{1}{m}\Psi_{i}^{\ast}(\omega_{FS})$ so that for any $i\geq1,$, there
exists a $G_{i}\in SL^{T}(N_{m}+1,\mathbf{C})$ such that
\[
\Psi_{i}=G_{i}\circ \Psi_{1}.
\]
At first as in the K\"{a}hler case (\cite{paul1}, \cite{tz}), the Sasaki
analogue of Mabuchi $K$-energy will have a lower bound on $\overline{\omega
}_{i}$ with \ a fixed $\overline{\omega}_{1}$
\[
K_{\eta_{0}}(\overline{\omega}_{1},\overline{\omega}_{i})\geq-C.
\]
By the $1$-cocycle condition of the transverse Mabuchi $K$-energy,%
\[
K_{\eta_{0}}(\omega_{0},\omega_{i})+K_{\eta_{0}}(\omega_{i},\overline{\omega
}_{i})=K_{\eta_{0}}(\omega_{0},\overline{\omega}_{i})=K_{\eta_{0}}(\omega
_{0},\overline{\omega}_{1})+K_{\eta_{0}}(\overline{\omega}_{1},\overline
{\omega}_{i})
\]
and then
\[
K_{\eta_{0}}(\omega_{0},\omega_{i})+K_{\eta_{0}}(\omega_{i},\overline{\omega
}_{i})\geq-C.
\]
Hence, to show (\ref{56}), we only need to get an upper bound for
\begin{equation}
K_{\eta_{0}}(\omega_{i},\overline{\omega}_{i})\leq C. \label{57}%
\end{equation}

For a fixed $m,$ we define
\[
\overline{\mathcal{\rho}}_{i}(x):=\frac{1}{m}\mathcal{\rho}_{t_{i}%
,m}(x):=\frac{1}{m}\mathcal{F}_{m}(x,t_{i}).
\]
Here $\mathcal{F}_{m}(x,t_{i})$ as in (\ref{58}). Then
\[
\omega_{i}=\overline{\omega}_{i}+\sqrt{-1}\partial_{B}\overline{\partial}%
_{B}\overline{\mathcal{\rho}}_{i}.
\]

It follows from \cite{t4} that the transverse Mabuchi $K$-energy has the
following explicit expression%
\begin{equation}%
\begin{array}
[c]{ccl}%
K_{\eta_{0}}(\omega_{i},\overline{\omega}_{i}) & = &
{\displaystyle \int \limits_{M}}
\log \frac{\overline{\omega}_{i}^{n}}{\omega_{i}^{n}}\overline{\omega}_{i}%
^{n}\wedge \eta_{0}+%
{\displaystyle \int \limits_{M}}
u(\overline{\omega}_{i}{}^{n}-\omega_{i}^{n})\wedge \eta_{0}\\
&  & -i\sum_{k=0}^{n-1}\frac{n-k}{n+1}%
{\displaystyle \int \limits_{M}}
(\partial_{B}\overline{\mathcal{\rho}}_{i}\wedge \overline{\partial}%
_{B}\overline{\mathcal{\rho}}_{i}\wedge \omega_{i}^{k}\wedge \overline{\omega
}_{i}^{n-k-1})\wedge \eta_{0}\\
& \leq &
{\displaystyle \int \limits_{M}}
\log \frac{\overline{\omega}_{i}^{n}}{\omega_{i}^{n}}\overline{\omega}_{i}%
^{n}\wedge \eta_{0}+%
{\displaystyle \int \limits_{M}}
u(\overline{\omega}_{i}{}^{n}-\omega_{i}^{n})\wedge \eta_{0}\\
& \leq &
{\displaystyle \int \limits_{M}}
\log \frac{\overline{\omega}_{i}^{n}}{\omega_{i}^{n}}\overline{\omega}_{i}%
^{n}\wedge \eta_{0}+C.
\end{array}
\label{57a}%
\end{equation}
Here $u$ is the transverse Ricci potential under the Sasaki-Ricci flow and we
have used the Perelman estimate that%
\[
|u(t_{i})|\leq C.
\]
On the other hand, it follows from (\ref{d1}) and (\ref{d11}) that%
\[
\overline{\omega}_{i}\leq C\omega_{i}.
\]
Therefore (\ref{57a}) implies (\ref{57}) and then we are done.
\end{proof}

\begin{corollary}
Let $(M,\xi,\eta_{0},g_{0})$ be a compact quasi-regular transverse Fano
Sasakian manifold of dimension up to seven and $(Z_{0}=M/\mathcal{F}_{\xi
},h_{0},\omega_{h_{0}})$ be the space of leaves of the characteristic
foliation which is a normal Fano projective K\"{a}hler orbifold with
codimension two orbifold singularities. If $M$ is transverse $K$-stable, then
under the Sasaki-Ricci flow, $M(t)$ converges to a compact transverse Fano
Sasakian manifold $M_{\infty}$ which is isomorphic to $M$ endowed with a
smooth Sasaki--Einstein metric.
\end{corollary}

\begin{proof}
Firstly, it follows from Corollary \ref{C52}, (\ref{54-b}) and (\ref{54}) that
$M_{\infty}$ must be Sasaki-Einstein. Moreover, the Lie algebra of all
Hamiltonian holomorphic vector fields is reductive (\cite{dt}, \cite{ber},
\cite{fow}). If $M_{\infty}$ is not equal to $M$, there is a family of the
weighted $\mathbf{C}^{\ast}$action $\{G(s)\}_{s\in C^{\ast}}\subset
SL^{T}(N_{m}+1,\mathbf{C})$ such that
\[
\Psi_{s}(M)=G(s)\circ \Psi_{1}(M)
\]
converges to the $S^{1}$-equivariant embedding of $M_{\infty}$ with respect to
the weighted $\mathbf{C}^{\ast}$action in $(\mathbf{CP}^{N_{m}},\omega_{FS}).$
Then the generalized Sasaki-Futaki invariant $f_{M_{\infty}}(V)$ of
$M_{\infty}$ vanishes as in (\ref{59}) and (\ref{59-1}).

On the other hand, by the assumption that $M$ is transverse $K$-stable, if
$M_{\infty}$ is not equal to $M$, then
\[
\operatorname{Re}f_{M_{\infty}}(V)>0.
\]
This is a contradiction. Hence $M_{\infty}=M$. Therefore there is a
Sasaki-Einstein metric on $M.$
\end{proof}


\begin{thebibliography}{99999}                                                                                            %


\bibitem[B]{b}D. Barden, Simply connected five-manifolds, Ann. of Math. (2) 82
(1965), 365-385.

\bibitem[BBEGZ]{bbegz}R. Berman, S. Boucksom, P. Eyssidieux, V. Guedj and A.
Zeriahi, K\"{a}hler-Einstein metrics and the K\"{a}hler-Ricci flow on log Fano
varieties, J. Reine Angew. Math. 751 (2019), 27-89.

\bibitem[Bel]{bel}F. A. Belgun, Normal CR structures on compact $3$-manifolds.
Math. Z. 238 (2001), no. 3, 441--460.

\bibitem[Ben]{ben}B. Chow, The Ricci flow on the 2-sphere, J. Differential
Geometry 33 (1991) 325-334.

\bibitem[Ber]{ber}B. Berndtsson, A Brunn-Minkowski type inequality for Fano
manifolds and some uniqueness theorems in K\"{a}hler geometry, Invent. math.
(2015) 200:149--200.

\bibitem[BG]{bg}C. P. Boyer and K. Galicki, Sasaki Geometry. Oxford
Mathematical Monographs. Oxford University Press, Oxford (2008).

\bibitem[BGS]{bgs}C. P. Boyer , K. Galicki and S. Simanca, Canonical Sasakian
metrics. Comm. Math. Phys. 279 (2008), no. 3, 705--733.

\bibitem[BM]{bm}S. Bando and T. Mabuchi, Uniqueness of Einstein K\"{a}hler
metrics modulo connected group actions, in: Algebraic geometry (Sendai 1985),
Adv. Stud. Pure Math. 10, North-Holland, Amsterdam (1987), 11--40.

\bibitem[Cao]{cao}H. Cao, Deformation of K\"{a}hler metrics to
K\"{a}hler-Einstein metrics on compact K\"{a}hler manifolds, Invent. Math.
81(1985), 359--372.

\bibitem[CC1]{cc1}J. Cheeger and T. H. Colding, Lower bounds on the Ricci
curvature and the almost rigidity of warped products, Ann. Math., 144 (1996), 189-237.

\bibitem[CC2]{cc2}J. Cheeger and T. H. Colding, On the structure of spaces
with Ricci curvature bounded below I, J. Diff. Geom., 46 (1997), 406-480.

\bibitem[CC3]{cc3}J. Cheeger and T. H. Colding, On the structure of spaces
with Ricci curvature bounded below II, J. Diff. Geom., 54 (2000), 13-35.

\bibitem[CCT]{cct}J. Cheeger, T. H. Colding and G. Tian,\textit{\ On the
singularities of spaces with bounded Ricci curvature}, Geom. Funct. Anal., 12
(2002), 873-914.

\bibitem[CDS1]{cds1}X. Chen, Simon Donaldson, and Song Sun,
K\"{a}hler-Einstein metrics on Fano manifolds. I, J. Amer. Math. Soc. 28
(2015), no. 1, 183--197.

\bibitem[CDS2]{cds2}X. Chen, Simon Donaldson, and Song Sun,
K\"{a}hler-Einstein metrics on Fano manifolds II, J. Amer. Math. Soc. 28
(2015), no. 1, 199--234.

\bibitem[CDS3]{cds3}X. Chen, Simon Donaldson, and Song Sun,
K\"{a}hler-Einstein metrics on Fano manifolds III, J. Amer. Math. Soc. 28
(2015), no. 1, 235--278.

\bibitem[CHLW]{chlw}S.-C. Chang, Y. Han, C. Lin and C.-T. Wu, Convergence of
the Sasaki-Ricci Flow on Sasakian $5$-Manifolds of General Type, arXiv:2203.00374.

\bibitem[CJ]{cj}T. Collins and A. Jacob, On the convergence of the
Sasaki-Ricci flow, Analysis, complex geometry, and mathematical physics: in
honor of Duong H. Phong, 11--21, Contemp. Math., 644, Amer. Math. Soc.,
Providence, RI, 2015.

\bibitem[CLW]{clw}S.-C. Chang, C. Lin and C.-T. Wu, Foliation divisorial
contraction by the Sasaki-Ricci flow on Sasakian $5$-manifolds, preprint.

\bibitem[CLW2]{clw2}S.-C. Chang, C. Lin and C.-T. Wu, On the existence of
conic Sasaki-Einstein metrics, in preparation.

\bibitem[CN]{cn}T. H. Colding and A. Naber, Sharp H\"{o}lder continuity of
tangent cones for spaces with a lower Ricci curvature bound and applications,
Ann. of Math., 176 (2012), 1173-1229.

\bibitem[Co1]{co1}T. Collins, The transverse entropy functional and the
Sasaki-Ricci flow, Trans. AMS., Volume 365, Number 3, March 2013, Pages 1277-1303.

\bibitem[Co2]{co2}T. Collins, Uniform Sobolev Inequality along the
Sasaki-Ricci Flow, J. Geom. Anal. 24 (2014), 1323--1336.

\bibitem[Co3]{co3}T. Collins, Stability and convergence of the Sasaki-Ricci
flow, J. reine angew. Math. 716 (2016), 1--27.

\bibitem[CSW]{csw}X, Chen, S. Sun and B. Wang, K\"{a}hler-Ricci flow,
K\"{a}hler-Einstein metric, and K-stability, Topol. 22 (2018) 3145-3173.

\bibitem[CT]{ct}T. Collins and V. Tosatti, K\"{a}hler currents and null loci,
Invent. math. (2015) 202, 1167-1198.

\bibitem[CZ1]{cz1}T. Collins and G. Szekelyhidi, K-semistability for irregular
Sasakian manifolds, J. Differential Geometry 109 (2018) 81-109.

\bibitem[CZ2]{cz2}T. Collins and G. Szekelyhidi, Sasaki-Einstein metrics and
K-stability, Geom. Topol. 23 (2019), no. 3, 1339--1413.

\bibitem[D]{d}S. K. Donaldson, Scalar curvature and stability of toric
varieties, J. Differential Geom. 62 (2002), 289-349.

\bibitem[DP]{dp}J. P. Demailly and M. Paun, Numerical characterization of the
K\"{a}hler cone of a compact K\"{a}hler manifold, Annals of Mathematics, 159
(2004), 1247-1274.

\bibitem[DK]{dk}J. P. Demailly and J. Kollar, Semi-continuity of complex
singularity exponents and K\"{a}hler-Einstein metrics on Fano manifolds, Ann.
Ec. Norm. Sup 34 (2001), 525-556.

\bibitem[DS]{ds}S. Donaldson and S. Sun, Gromov-Hausdorff limits of K\"{a}hler
manifolds and algebraic geometry, Acta Math. 213(1) (2014) 63--106.

\bibitem[DT]{dt}W. Ding and G. Tian, K\"{a}hler-Einstein metrics and the
generalized Futaki invariants. Invent. Math., 110 (1992), 315-335.

\bibitem[EKA]{eka}A. El Kacimi-Alaoui, Operateurs transversalement elliptiques
sur un feuilletage riemannien et applications, Compos. Math. 79 (1990) 57--106.

\bibitem[F]{f}A. Futaki, An obstruction to the existence of Einstein
K\"{a}hler metrics, Invent. Math. 73 (1983), 437--443.

\bibitem[FOW]{fow}A. Futaki, H. Ono and G.Wang, Transverse K\"{a}hler geometry
of Sasaki manifolds and toric Sasaki--Einstein manifolds, J. Differential
Geom. 83 (2009) 585--635.

\bibitem[Gei]{gei}H. Geiges, Normal Contact Structures on $3$-manifolds,
Tohoku Math. J. 49 (1997), 415-422.

\bibitem[GMSW]{gmsw}J. P. Gauntlett, D. Martelli, J. Sparks and D. Waldram,
Sasaki-Einstein Metrics on $\mathbf{S}^{2}\times \mathbf{S}^{3}$, Adv. Theor.
Math. Phys. 8 (2004), 711--734.

\bibitem[GKN]{gkn}M. Godlinski, W. Kopczynski and P. Nurowski, Locally
Sasakian manifolds, Classical Quantum Gravity 17 (2000) L105--L115.

\bibitem[H1]{h1}R. S. Hamilton, The Ricci flow on surfaces, Math, and General
Relativity, Contemporary Math. 71 (1988), 237-262.

\bibitem[H2]{h2}R. S. Hamilton, Three-manifolds with positive Ricci curvature,
J. Differential Geom. 17 (1982), no. 2, 255--306.

\bibitem[H3]{h3}R.S. Hamilton, The formation of singularities in the Ricci
flow, in Surveys in differential geometry, Vol. II (Cambridge, MA, 1993),
7--136, Int. Press, Cambridge, MA, 1995.

\bibitem[He]{he}W. He, The Sasaki-Ricci flow and compact Sasaki manifolds of
positive transverse holomorphic bisectional curvature, J. Geom. Anal. 23
(2013), 1876-931.

\bibitem[HLM]{hlm}C.-Y. Hsiao, X. Li and G. Marinescu, Equivariant Kodaira
embedding for CR manifolds with circle action, Michigan Math. J. 70 no.1
(2021), 55-113.

\bibitem[K1]{k1}J. Kollar, Singularities of pairs, Algebraic geometry, Santa
Cruz 1995, Proc. Sympos. Pure Math., vol. 62, Amer. Math. Soc., Providence, RI
(1997) 221-287.

\bibitem[K2]{k2}J. Kollar,\textit{\ }Einstein metrics on connected sums of
$\mathbf{S}^{2}\times \mathbf{S}^{3}$, J. Differential Geom. 75 (2007), no. 2, 259--272.

\bibitem[K3]{k3}J. Kollar, Einstein metrics on five-dimensional Seifert
bundles, J. Geom. Anal. 15 (2005), no. 3, 445--476.

\bibitem[LT]{lt}F. Luo and G. Tian, Liouville equation and spherical convex
polytopes, Proc Amer. Math. Soc. 116 (1992), no. 4, 1119-1129.

\bibitem[LZ]{lz}J. Liu and X. Zhang, The conical K\"{a}hler-Ricci flow on Fano
manifolds, Advances in Mathematics, 307(2017), 1324--1371.

\bibitem[M]{m}T. Mabuchi, K-energy maps integrating Futaki invariants, Tohoku
Math. J., 38, 245-257 (1986).

\bibitem[Mo]{mo}N. Mok, The uniformization theorem for compact K\"{a}hler
manifolds of nonnegative holomorphic bisectional curvature. J. Differ. Geom.
27(2) (1988), 179--214.

\bibitem[MSY]{msy}Dario Martelli, James Sparks and Shing--Tung Yau,
Sasaki--Einstein manifolds and volume minimisation, Communications in
Mathematical Physics, 280 (2008), 611--673

\bibitem[Na]{na}A. M. Nadel, Multiplier ideal sheaves and K\"{a}hler-Einstein
metrics of positive scalar curvature, Ann. of Math. (2) 132 (1990), no. 3, 549-596.

\bibitem[NS]{ns}Y. Nitta and K. Sekiya, A diameter bound for Sasaki manifolds
with applications to uniqueness for Sasaki-Einstein structure, preprint 2009, http://arxiv.org/abs/0906.0170.

\bibitem[NT]{nt}S. Nishikawa and P. Tondeur, Transversal infinitesimal
automorphisms for harmonic K\"{a}hler foliation, Tohoku Math. J., 40(1988), 599-611.

\bibitem[P1]{p1}G. Perelman, The entropy formula for the Ricci flow and its
geometric applications, preprint, arXiv: math.DG/0211159.

\bibitem[P2]{p2}G. Perelman, Ricci flow with surgery on three-manifolds,
preprint, arXiv: math.DG/0303109.

\bibitem[P3]{p3}G. Perelman, Finite extinction time for the solutions to the
Ricci flow on certain three-manifolds, preprint, arXiv: math.DG/0307245.

\bibitem[Paul1]{paul1}S. T. Paul, Hyperdiscriminant polytopes, Chow polytopes,
and Mabuchi energy asymptotics, Ann. Math., 175 (2012), 255-296.

\bibitem[Paul2]{paul2}S. T. Paul, A numerical criterion for K-energy maps of
algebraic manifolds, arXiv:1210.0924v1.

\bibitem[Pe]{pe}P. Petersen, Convergence theorems in Riemannian geometry, in
\textquotedblright Comparison Geometry\textquotedblright \ edited by K. Grove
and P. Petersen, MSRI Publications, vol 30 (1997), Cambridge Univ. Press, 167-202.

\bibitem[PSSW]{pssw}D. H. Phong, J. Song, J. Sturm and X. Wang, The Ricci-Flow
on the Sphere with Marked Points, J. Differential Geometry 114 (2020) 117-170.

\bibitem[PSSWe]{psswe}D. H. Phong, J. Song, J. Sturm and B. Weinkove, The
K\"{a}hler-Ricci flow and $\overline{\partial}$-operator on vector fields, J.
Differ. Geom., 81 (2009), 631-647.

\bibitem[PW1]{pw1}P. Petersen and G.F. Wei,\textit{\ }Relative volume
comparison with integral curvature bounds, Geom. Funct. Anal., 7 (1997), 1031-1045.

\bibitem[PW2]{pw2}P. Petersen and G.F. Wei, Analysis and geometry on manifolds
with integral Ricci curvature bounds,\ II, Trans. AMS., 353 (2001), 457-478.

\bibitem[RT]{rt}J. Ross and R.P. Thomas, Weighted projective embeddings,
stability of orbifolds and constant scalar curvature K\"{a}hler metrics, JDG
88 (2011), no. 1, 109-160.

\bibitem[Ru]{ru}P. Rukimbira, Chern-Hamilton's conjecture and K-contactness,
Houston J. Math. 21 (1995), no. 4, 709-718.

\bibitem[S]{s}S. Smale, On the structure of 5-manifolds\textit{,} Ann. of
Math. (2) 75 (1962), 38-46.

\bibitem[Sp]{sp}James Sparks, Sasaki-Einstein Manifolds, Surveys in
Differential Geometry 16 (2011), 265-324.

\bibitem[Shi]{shi}W.-X. Shi, Ricci deformation of the metric on complete
noncompact Riemannian manifolds, J. Diff. Geom., 30 (1989), 303-394.

\bibitem[ST]{st}N. Sesum and G. Tian, Bounding scalar curvature and diameter
along the K\"{a}hler Ricci Flow (after Perelman), J. Inst. of Math. Jussieu, 7
(2008), no. 3, 575-587.

\bibitem[SWZ]{swz}K. Smoczyk, G. Wang and Y. Zhang, The Sasaki-Ricci flow,
Internat. J. Math. 21 (2010), no. 7, 951--969.

\bibitem[T1]{t1}G. Tian, On Calabi's conjecture for complex surfaces with
positive first Chern class, Invent. Math., 101, (1990), 101-172.

\bibitem[T2]{t2}G. Tian, Partial $C^{0}$-estimates for K\"{a}hler-Einstein
metrics, Commun. Math. Stat., 1 (2013), 105-113.

\bibitem[T3]{t3}G. Tian, Kahler-Einstein metrics with positive scalar
curvature, Invent. Math. 137 (1997), no. 1, 1-37.

\bibitem[T4]{t4}G. Tian, Canonical Metrics in K\"{a}hler Geometry, Lectures in
Mathematics ETH Z\={u}rich, Birkh\u{u}user Verlag, Basel, 2000.

\bibitem[T5]{t5}G. Tian, K-stability and K\"{a}hler-Einstein metrics, Comm.
Pure Appl. Math. 68 (2015), no. 7, 1085--1156. Corrigendum: Comm. Pure Appl.
Math. 68 (2015), no. 11, 2082--2083.

\bibitem[Tro]{tro}M. Troyanov, Prescribing curvature on compact surfaces with
conical singularities, Trans. Amer. Math. Soc. 324 (1991), no. 2, 793-821.

\bibitem[TW]{tw}G. Tian and F. Wang, On the existence of conic
K\"{a}hler-Einstein metrics. Adv. Math. 375 (2020), 107413, 42 pp.

\bibitem[TZ]{tz}G. Tian and Z. Zhang, Regularity of K\"{a}ler-Ricci flow on
Fano manifolds, Acta Math. 216 (2016) No. 1, 127-176.

\bibitem[TZhu1]{tzhu1}G. Tian and X.-H. Zhu, Convergence of K\"{a}hler-Ricci
flow, J. Amer. Math. Soc. 20 (2007), 675-699.

\bibitem[TZhu2]{tzhu2}G. Tian and X.-H. Zhu, Convergence of K\"{a}hler-Ricci
flow on Fano manifolds II, J. Reine Angew. Math.. 678 (2013), 223-245.

\bibitem[Wu]{wu}L.-F. Wu, The Ricci-flow on $2$-orbifolds with positive
curvature, J. Diff. Geom. 44 (1991), no 2, 575-596.

\bibitem[WZ]{wz}G. Wang and Y. Zhang, The Sasaki--Ricci flow on Sasakian
$3$-spheres, Commun. Math. Stat. 1 (2013), no. 1, 43--71.

\bibitem[Y1]{y1}S.-T. Yau, On the Ricci curvature of a compact K\"{a}hler
manifold and the complex Monge-Amp\`{e}re equation,\textit{\ I}, Comm. Pure.
Appl. Math. 31 (1978), 339-411.

\bibitem[Y2]{y2}S.-T. Yau, Open problems in geometry, Proc. Symp. Pure Math.
54 (1993), 1-18.

\bibitem[Zh1]{zh1}X. Zhang, Some invariants in Sasakian geometry,
International Mathematics Research Notices, Volume 2011, Issue 15 (2011), 3335--3367.

\bibitem[Zh2]{zh2}X. Zhang, Energy Properness and Sasakian-Einstein metrics,
Commun. Math. Phys. 306 (2011), 229--260.
\end{thebibliography}
\end{document}